\documentclass[a4paper,11pt,leqno]{article}

\usepackage[matrix,arrow,tips,curve]{xy}
\usepackage[french,english]{babel}
\usepackage{amsmath}
\usepackage{amssymb,amsthm}
\usepackage{enumerate}

\newtheorem{thm}{Theorem}[section]
\newtheorem{lemma}[thm]{Lemma}
\newtheorem{proposition}[thm]{Proposition}
\newtheorem{corollary}[thm]{Corollary}

\newtheorem{thmdue}{Theorem}[subsection]
\newtheorem{lemma2}[thmdue]{Lemma}
\newtheorem{proposition2}[thmdue]{Proposition}
\newtheorem{corollary2}[thmdue]{Corollary}

\newtheorem{claim2}[thmdue]{Claim}

\theoremstyle{definition}

\newtheorem{example}[thm]{Example}
\newtheorem{parg}[thm]{}

\newtheorem{remark2}[thmdue]{Remark}

\newtheorem{parg2}[thmdue]{}

\renewcommand{\theequation}{\thethm}

\newcommand{\ph}{\varphi}
\newcommand{\pr}{\mathbb{P}}
\newcommand{\Q}{\mathbb{Q}}
\newcommand{\R}{\mathbb{R}}
\newcommand{\N}{\mathcal{N}_1}
\newcommand{\Supp}{\operatorname{Supp}}
\newcommand{\Sing}{\operatorname{Sing}}
\newcommand{\NE}{\operatorname{NE}}
\newcommand{\Exc}{\operatorname{Exc}}
\newcommand{\Lo}{\operatorname{Locus}}
\newcommand{\codim}{\operatorname{codim}}
\newcommand{\dom}{\operatorname{dom}}
\newcommand{\Hilb}{\operatorname{Hilb}}
\newcommand{\Chow}{\operatorname{Chow}}
\newcommand{\Nef}{\operatorname{Nef}}
\newcommand{\Eff}{\operatorname{Eff}}

\newlength{\Mheight}
\newlength{\cwidth}
\title{On the 
Picard number of divisors in Fano manifolds}
\author{Cinzia Casagrande}
\date{May 24, 2011}

\begin{document}
\selectlanguage{english}
\maketitle
\begin{abstract}
\noindent 
Let $X$ be a complex Fano manifold of arbitrary
  dimension, and $D$ a 
 prime divisor in $X$. We consider the image $\N(D,X)$ of $\N(D)$ in 
$\N(X)$ under
 the natural push-forward of 1-cycles. We show that 
$\rho_X-\rho_D\leq\codim\N(D,X)\leq 8$.
Moreover if $\codim\N(D,X)\geq 3$, then either $X\cong S\times T$ where
 $S$ is a Del Pezzo surface, or $\codim\N(D,X)=3$ 
and $X$ has a fibration in Del Pezzo
 surfaces onto a Fano manifold $T$ such that $\rho_X-\rho_T=4$. 
\end{abstract} 
\selectlanguage{french}
\begin{abstract}
\noindent Soit $X$ une vari\'et\'e de Fano lisse et complexe de
  dimension arbitraire,  
et $D$ un diviseur premier dans $X$. Nous consid\'erons 
l'image $\N(D,X)$ de $\N(D)$
 dans $\N(X)$ par l'application naturelle de
``push-forward'' de $1$-cycles. Nous d\'emontrons que 
$\rho_X-\rho_D\leq\codim\N(D,X)\leq 8$. De plus, si
 $\codim\N(D,X)\geq 3$, alors soit  $X\cong S\times T$ o\`u $S$ est une
 surface de Del Pezzo, soit $\codim\N(D,X)= 3$ et
$X$ a une fibration en surfaces de Del
 Pezzo sur une  vari\'et\'e de Fano lisse $T$, telle que
 $\rho_X-\rho_T=4$. 
\end{abstract}
\selectlanguage{english}
{\renewcommand{\thefootnote}{}
}
{\footnotesize\tableofcontents}
\section{Introduction}
Let $X$ be a complex Fano manifold of arbitrary dimension $n$, and
consider a prime divisor $D\subset X$.  
We denote by $\N(X)$ the $\R$-vector space of one-cycles in $X$, with
real coefficients, modulo numerical equivalence; its dimension is the
\emph{Picard number} $\rho_X$ of $X$, and similarly for $D$.
The inclusion $i\colon D\hookrightarrow X$ induces a push-forward of
one-cycles
$i_*\colon\N(D)\to\N(X)$, that does not need to be
injective nor surjective. We are interested in the image
 $$\N(D,X):=i_*(\N(D))\subseteq\N(X),$$ which is  
the linear subspace of $\N(X)$ spanned by
numerical classes of curves contained in $D$. 
The codimension of $\N(D,X)$ in $\N(X)$
 is equal to the dimension of the kernel of the restriction
$H^2(X,\R)\to H^2(D,\R)$.

If $X$ is a surface, then it follows from the classification of 
Del Pezzo surfaces that $\codim\N(D,X)=\rho_X-1\leq 8$.
Our main result is that the same holds in any dimension.
\begin{thm}\label{product}
Let $X$ be a Fano manifold of dimension $n$. 
For every prime divisor $D\subset X$,  we have 
$$\rho_X-\rho_D\leq \codim\N(D,X)\leq 8.$$
Moreover, suppose that there exists a prime divisor $D$ with
$\codim\N(D,X)\geq 3$. Then one of the following holds:
\begin{enumerate}[$(i)$]
\item $X\cong
S\times T$, where $S$ is a Del Pezzo surface
with $\rho_S\geq\codim\N(D,X)+1$, 
 and $D$ dominates $T$ under the projection;
\item  $\codim\N(D,X)=3$ and there exists a flat
surjective morphism  
$\ph\colon X\to T$, with connected fibers,
where $T$ is an
  $(n-2)$-dimensional Fano manifold, and $\rho_X-\rho_T=4$.
\end{enumerate}
\end{thm}
When $n\geq 4$ and $D$ is ample, one has
$\N(D,X)=\N(X)$ and also $\dim\N(D,X)=\rho_D$ by Lefschetz Theorems on
hyperplane sections, see \cite[Example 3.1.25]{lazI}.
However in general $\dim\N(D,X)$ can be smaller than $\rho_X$:
for instance, if $D\cong\pr^{n-1}$ is the exceptional divisor of 
the blow-up $X$ of any projective manifold at a point, we have
$\rho_D=\dim\N(D,X)=1<\rho_X$.

In case $(ii)$ of Theorem \ref{product} the variety $X$ does not need to
be a product of lower dimensional varieties, see Example~\ref{cod3}.

Theorem \ref{product}
generalizes an analogous result  in \cite{fano} for toric
Fano varieties, obtained in a completely different way,
using combinatorial techniques.

We recall that
 the pseudo-index of a
Fano manifold $X$ is
$$\iota_X=\min\{-K_X\cdot C\,|\,C\text{ is a rational curve 
in }X\},$$
and is a multiple of the index of $X$;
one expects that Fano manifolds with large pseudo-index are simpler. 
When $\iota_X>1$ (\emph{i.e.}\ when $X$ 
does not contain rational curves of anticanonical degree one),  
 we show a stronger version of 
Theorem \ref{product}.
\begin{thm}\label{corindex}
Let $X$ be a Fano manifold with pseudo-index $\iota_X>1$. 
For every prime divisor $D\subset X$,  we have 
$\codim\N(D,X)\leq 1$. More precisely, one of
the following holds:
\begin{enumerate}[$(i)$]
\item $\iota_X=2$ and there exists
 a smooth morphism $\ph\colon X\to Y$
with fibers isomorphic to $\pr^1$,
 where $Y$ is
 a Fano manifold with
  $\iota_Y>1$; 
\item for every prime divisor $D\subset X$, we have $\N(D,X)=\N(X)$,
$\rho_X\leq\rho_D$,
  and the restriction $H^2(X,\R)\to H^2(D,\R)$ is
  injective. Moreover
for every pair of prime divisors $D_1,D_2$ in $X$, we have $D_1\cap
  D_2\neq\emptyset$.
\end{enumerate}
\end{thm}
The author was led to this subject by 
 the study of
Fano manifolds with large Picard number 
(see \cite{fanos} for an account of this problem). 
Let us mention two straightforward 
consequences of Theorem \ref{product}, which give bounds 
on $\rho_X$ in some good situations. 
The first concerns the case $\dim X\leq 5$, 
while the second is about Fano manifolds having a morphism onto a curve.
\begin{corollary}\label{lorenzo}
Let $X$ be a Fano manifold, and suppose that there exists a prime
divisor $D\subset X$ such that
$\codim\N(D,X)\geq 3$.

If $\dim X=4$ then either $\rho_X\leq 6$, or $X$ is a
product of Del Pezzo surfaces and $\rho_X\leq 18$.

If $\dim X=5$ then either $\rho_X\leq 9$, or $X$ is a
product and $\rho_X\leq 19$.
\end{corollary}
\begin{corollary}\label{p1}
Let $X$ be a Fano manifold, $\ph\colon X\to \pr^1$ a surjective morphism 
with connected fibers, and
$F\subset X$ a general fiber. Then $\rho_X\leq\rho_F+8$. 

Moreover if $\rho_X\geq \rho_F+4$, then $X\cong
S\times T$ where $S$ is a Del Pezzo surface,  
$\ph$ factors through the projection $X\to S$, and $F\cong\pr^1\times T$.
\end{corollary}
Finally, we notice that some of the properties given by
 Theorem \ref{product} are inherited by
 varieties dominated by a Fano manifold. We give two
applications, and refer the reader to Lemma
\ref{corollZ} for a more general statement.
\begin{corollary}\label{divtocurve}
Let $X$ be a Fano manifold and $\ph\colon X\to Y$ a
surjective morphism. Suppose that there exists a prime divisor $D\subset X$ 
such that $\dim\ph(D)\leq 1$ (this always holds if $\dim Y=2$).
Then $\rho_Y\leq 9$.

Moreover if $\rho_Y\geq 5$ then $\dim Y\leq 3$ and
$X\cong
S\times T$, where $S$ is a Del Pezzo surface.
\end{corollary}
\begin{corollary}\label{3folds}
Let $X$ be a Fano manifold and $\ph\colon X\to Y$ a surjective
morphism 
with $\dim Y=3$. Then $\rho_Y\leq 10$.

Moreover
if $\rho_Y\geq 6$ then $X\cong
S\times T$  
where $S$ is a Del Pezzo surface, $T$ has a contraction
onto $\pr^1$, and $\ph$ factors 
  through $X\to S\times\pr^1$.
\end{corollary}

\medskip

\noindent\textbf{Outline of the paper.}
The idea that a special divisor should affect the geometry of  $X$ is
classical. In \cite{bonwisncamp} Fano manifolds containing a
divisor $D\cong\pr^{n-1}$ with normal bundle
$\mathcal{N}_{D/X}\cong\mathcal{O}_{\pr^{n-1}}(-1)$ are classified. This
classification
has been
extended in \cite{toru} to the case
$\mathcal{N}_{D/X}\cong\mathcal{O}_{\pr^{n-1}}(-a)$ with $a>0$; moreover 
\cite[Proposition 5]{toru} 
shows that if $X$ contains a divisor $D$ with $\rho_D=1$, then
$\rho_X\leq 3$. 
More generally, divisors $D\subset X$ with $\dim\N(D,X)=1$ or $2$ play
an important role in \cite{fanos,31}. 

In section~\ref{Running} we treat the main construction that will be
used in the paper, based on the analysis of a Mori program for $-D$,
where $D\subset X$ is a prime divisor;
this is a development of a technique
 used in \cite{31}.
Let us give 
an idea of our approach, referring the reader to
section~\ref{Running} for more details.

 After  \cite{BCHM,hukeel}, we know that we can run
a Mori program for any divisor in a  Fano 
manifold $X$. In fact we need to consider  \emph{special Mori
  programs}, where all involved extremal rays have positive
intersection with the anticanonical divisor (see section
\ref{Special}). 

Then, given a prime divisor $D\subset X$, 
we consider a special Mori program for $-D$, which roughly means that we
contract or flip extremal rays having positive
intersection with $D$,
until we get a fiber type
contraction such that (the transform of) $D$ dominates the target.
  
If $c:=\codim\N(D,X)>0$,
 by studying how the codimension of $\N(D,X)$ varies under
the birational maps and the related properties of the extremal
rays,
we obtain $c-1$ \emph{pairwise disjoint} prime divisors
$E_1,\dotsc,E_{c-1}\subset X$, all intersecting
$D$, such that each $E_i$ is a smooth $\pr^1$-bundle with
$E_i\cdot f_i=-1$, where $f_i\subset E_i$ is a fiber (see
Proposition \ref{MP2} and Lemma~\ref{pluto}). 
We call $E_1,\dotsc,E_{c-1}$ the $\pr^1$-bundles determined by the
special Mori program for $-D$ that we are considering; they play an
essential role throughout the paper. 
 
We conclude section~\ref{Running} 
proving Theorem \ref{corindex} about the case with pseudo-index $\iota_X>1$.

In section~\ref{Divisors} we consider the following invariant of $X$:
$$c_X:=\max\{\codim\N(D,X)\,|\,D\text{ is a prime divisor in }X\}.$$ 
In terms of this invariant, our main 
result is that $c_X\leq 8$, and if $c_X\geq 3$,
then either $X$ is a product, or $c_X=3$ and $X$ has a flat fibration onto an
$(n-2)$-dimensional Fano manifold (see Theorem \ref{gen} 
for a precise statement).
The proof of this result is quite long: it takes the 
whole section~\ref{Divisors},
and is 
divided in several steps; see \ref{primooutline} for a plan.
The strategy is to apply the construction of section~\ref{Running} to 
prime divisors of ``minimal Picard number'',
\emph{i.e.}\ with $\codim\N(D,X)=c_X$.
We show that there exists a prime divisor $E_0$ 
with $\codim\N(E_0,X)=c_X$, such that $E_0$ is a smooth $\pr^1$-bundle with 
$E_0\cdot f_0=-1$, where $f_0\subset E_0$ is a fiber.
Applying the previous results to $E_0$, we
 obtain a bunch of disjoint divisors with a
$\pr^1$-bundle structure, and
we use them to show that $X$ is a product, or 
to construct a fibration in Del Pezzo surfaces. 

Finally in section~\ref{quarta} we 
use this result (Theorem \ref{gen}) to prove the remaining results stated 
above:  Theorem \ref{product} and its Corollaries \ref{lorenzo} 
to \ref{3folds}.

\medskip

\noindent\textbf{Acknowledgements.} I am grateful to Tommaso de Fernex for an
important suggestion concerning Proposition~\ref{tomm}. 

This paper was written mainly during a
visit to the Mathematical Sciences Research Institute in Berkeley,
for the program in Algebraic Geometry in spring 2009. I
would like to thank MSRI for the kind hospitality, and 
GNSAGA-INdAM 
and the Research Project ``Geometria delle variet\`a 
algebriche e dei loro spazi di moduli'' (PRIN 2006) for financial support. 

\medskip

\noindent\textbf{Notation and terminology}\\
We work over the field of complex numbers.\\
A \emph{manifold} is a smooth variety.\\
A \emph{$\pr^1$-bundle} is a projectivization of a rank $2$ vector bundle.

\smallskip

\noindent Let $X$ be a projective variety.\\
$\N(X)$  (respectively, $\mathcal{N}^1(X)$) is the $\R$-vector 
space of one-cycles (respectively, Cartier divisors) with
real coefficients, modulo numerical equivalence.\\
$[C]$ is the numerical equivalence class in $\N(X)$ of a curve
$C\subset X$; 
$[D]$ is the numerical equivalence class in $\mathcal{N}^1(X)$ of a
$\Q$-Cartier divisor $D$ in $X$.\\
If $E\subset X$ is an irreducible closed subset and $C\subset E$ is a
curve, $[C]_E$ is the numerical equivalence class of $C$ in $\N(E)$.\\
$\equiv$ stands for numerical equivalence (for both 1-cycles and
$\Q$-Cartier divisors).\\ 
For any $\Q$-Cartier divisor $D$ in $X$,
$D^{\perp}:=\{\gamma\in\N(X)\,|\,D\cdot\gamma=0\}$.\\
$\NE(X)\subset\N(X)$ is the convex cone generated by classes of
effective curves, and $\overline{\NE}(X)$ is its closure.\\
An \emph{extremal ray} $R$ of $X$ is a one-dimensional
face of $\overline{\NE}(X)$; $\Lo(R)\subseteq X$ is the union of all
curves whose class is in $R$.\\ 
If $R$ is an extremal ray of $X$ and $D$ is a $\Q$-Cartier divisor in $X$,
we say that $D\cdot R>0$, respectively $D\cdot R=0$, etc.\ if 
for $\gamma\in R\smallsetminus\{0\}$ we have $D\cdot \gamma>0$, 
respectively $D\cdot \gamma=0$, etc. 

\smallskip

\noindent Assume that $X$ is normal.\\
 A \emph{contraction} of $X$  
is a surjective morphism with connected fibers $\ph\colon X\to Y$, where
$Y$ is normal
and projective;
$\NE(\ph)$ is the face of $\overline{\NE}(X)$ generated by classes of curves
contracted by $\ph$.\\
A contraction $\ph\colon X\to Y$ is \emph{elementary} if
$\rho_X-\rho_Y=1$.\\
We say that an elementary contraction $\ph\colon X\to Y$ (or the extremal ray
$\NE(\ph)$)
is of type $(n-1,n-2)^{sm}$ if
it is the blow-up of a smooth codimension $2$ subvariety
contained in the smooth locus of $Y$ (here $n=\dim X$). 

\smallskip

\noindent If $Z\subseteq X$ is a closed subset and 
$i\colon Z\hookrightarrow X$ is the inclusion,
  we set
$$\N(Z,X):=i_*(\N(Z))\subseteq\N(X)\ \text{ and
}\ \NE(Z,X):=i_*(\NE(Z))\subseteq\NE(X)\subset\N(X).$$ 
\section{Mori programs and prime divisors}\label{Running}
\subsection{Special Mori programs in Fano manifolds}\label{Special}
In this section we recall what a Mori program is, and explain that by
\cite{hukeel} and \cite{BCHM} we can run  a Mori program for any
divisor on a Fano manifold. We also introduce and show the existence
of  ``special Mori programs'', where all involved extremal rays have
positive intersection with the anticanonical divisor. 

We begin by recalling
 the following fundamental result.
\begin{thm}[\cite{BCHM}, Corollary 1.3.2]\label{FF}
Any Fano manifold is a Mori dream space.
\end{thm} 
\noindent We refer the reader to \cite{hukeel} 
for the definition and properties
of a Mori dream spaces; in particular, 
 a Mori dream space is always a normal and
$\Q$-factorial projective variety. 
We also need
the following.
\begin{proposition}[\cite{hukeel}, Proposition 1.11(1)]\label{MoriProgram}
Let $X$ be a Mori dream space and $B$ a divisor in $X$. Then there exists
a finite sequence
\stepcounter{thm}
\begin{equation}\label{MP}
X=X_0\stackrel{\sigma_0}{\dasharrow} X_1 \dasharrow\quad\cdots
\quad
\dasharrow X_{k-1}\stackrel{\sigma_{k-1}}{\dasharrow} X_k
\end{equation}
such that:
\begin{enumerate}[$\bullet$]
\item every $X_i$ is a normal and $\Q$-factorial projective variety; 
\item for every $i=0,\dotsc,k-1$ there exists an extremal
  ray $Q_i$ of $X_i$ such that $B_i\cdot
  Q_i<0$, where $B_i\subset X_i$ is the transform\footnote{More precisely,
$B_i$ is the transform of $B_{i-1}$ if $\sigma_{i-1}$ is
a flip, and $B_i=(\sigma_{i-1})_*(B_{i-1})$ if $\sigma_{i-1}$ is a
divisorial contraction.} of $B$, $\Lo(Q_i)\subsetneq X_i$, and 
  $\sigma_i$ is either the contraction of $Q_i$ (if $Q_i$ is
  divisorial), or its flip (if $Q_i$ is small); 
\item either $B_k$ is nef, or
there exists an extremal ray $Q_k$ in $X_k$, with a fiber type
  contraction $\ph\colon X_k\to Y$,
such
  that $B_k\cdot
  Q_k<0$.
\end{enumerate}
Moreover, the choice of the extremal rays $Q_i$ is arbitrary among
those that have negative intersection with $B_i$.
\end{proposition}
\noindent A sequence as above is called a {\bf Mori program} for the
  divisor $B$. We refer the reader to \cite[Definition 6.5]{kollarmori} for
  the definition of flip.

An important remark is that when $X$ is Fano, there is always a
suitable choice of a Mori program where all involved extremal rays
have positive intersection with the anticanonical divisor. 
\begin{proposition}\label{tomm}
Let $X$ be a Fano manifold and $B$ a divisor on $X$.
Then there exists a Mori program for $B$ as \eqref{MP},
such that $-K_{X_i}\cdot
Q_i>0$ for every $i=0,\dotsc,k$.
We call such a sequence a
{\bf special Mori program} for $B$.
\end{proposition}
This is a very special case of the MMP with scaling, see
\cite[Remark 3.10.9]{BCHM}. For the reader's convenience, we give a
proof. The idea is to choose a facet of the cone of nef divisors
$\Nef(X)\subset\mathcal{N}^1(X)$
met by moving from $[B]$ to $[-K_X]$ along a line in $\mathcal{N}^1(X)$, and
to repeat the same at each step. 
\begin{proof}[Proof of Proposition \ref{tomm}]
By Theorem \ref{FF} $X$ is a Mori dream space, therefore Proposition
\ref{MoriProgram} 
applies to $X$, and there exists a Mori program for $B$. 
We have to prove that we can choose $Q_0,\dotsc,Q_k$ with $B_i\cdot Q_i<0$
and
 $-K_{X_i}\cdot Q_i>0$ for all $i=0,\dotsc,k$.

We can assume that $B$ is not nef. Set
$$\lambda_0:=\sup\{\lambda\in\R\,|\,(1-\lambda)(-K_X)+\lambda B\text{
  is nef}\},$$
so that $\lambda_0\in\Q$, $0<\lambda_0<1$, and
$H_0:=(1-\lambda_0)(-K_X)+\lambda_0 B$ is nef but not ample. 
Then there exists an extremal
ray $Q_0$ of $\NE(X)$ such that $H_0\cdot Q_0=0$ and $B\cdot Q_0<0$;
in particular, $-K_X\cdot Q_0>0$.  

If $Q_0$ is of fiber type, we are done. Otherwise, let $\sigma_0\colon
X\dasharrow X_1$ be either the contraction of $Q_0$ (if divisorial), 
or its flip (if small), and let $B_1$ be the transform of $B$. Then 
$(1-\lambda_0)(-K_{X_1})+\lambda_0 B_1$ is nef in $X_1$.

If $B_1$ is nef we are done. If not, we set
$$\lambda_1:=\sup\{\lambda\in\R\,|\,(1-\lambda)(-K_{X_1})+\lambda B_1\text{
  is nef}\},$$
so that $\lambda_1\in\Q$, $\lambda_0\leq\lambda_1<1$, and
$H_1:=(1-\lambda_1)(-K_{X_1})+\lambda_1 B_1$ is nef but not ample. 
There exists an extremal
ray $Q_1$ of $\NE(X_1)$ such that $H_1\cdot Q_1=0$ and $B_1\cdot
Q_1<0$, hence $-K_{X_1}\cdot Q_1>0$. Now we iterate the procedure.
\end{proof}
\subsection{Running a Mori program for $-D$}\label{run}
In this section we study in detail what happens when we run a Mori
program for $-D$, where $D$ is a prime divisor. This point of view has
already been considered in \cite{31}, and is somehow opposite to the
classical one: we consider extremal rays having \emph{positive} 
intersection with $D$. 
In particular, we are interested in how the number $\codim\N(D,X)$
varies under the Mori program. 

We first describe the general situation for a prime divisor $D$ in a
Mori dream space (Lemma~\ref{general}), and then consider the case of
a \emph{special} Mori program for $-D$ where $D$ is a prime divisor in  
 a Fano manifold
 (Lemma~\ref{pluto}). In particular, we will show the following.
\begin{proposition}\label{MP2}
Let $X$ be a Fano manifold and $D\subset X$ a prime
divisor. Suppose that $\codim\N(D,X)>0$. 

Then there exist
pairwise disjoint smooth prime divisors $E_1,\dotsc,E_s\subset X$, with
$s=\codim\N(D,X)-1$ or $s=\codim\N(D,X)$,
such that
 every $E_j$ is a $\pr^1$-bundle with  $E_j\cdot
f_j=-1$, where  $f_j\subset E_j$ is a fiber;
moreover $D\cdot f_j>0$ and $[f_j]\not\in\N(D,X)$. In particular
$E_j\cap D\neq\emptyset$ and $E_j\neq D$.
\end{proposition}
\noindent It is important to point out that
 the $\pr^1$-bundles $E_1,\dotsc,E_s$ are determined
 not only by $D$, but
by the choice of a special Mori program for $-D$ (see
 Lemma~\ref{pluto}). In fact the divisors $E_j$ are the transforms of
 the loci of some of the extremal rays of the Mori program, the ones
 where $\codim\N(D,X)$ drops. 

Finally we study in more detail the case where $s=\codim\N(D,X)-1$ in the 
Proposition above; in this situation
we show that there is an open subset of $X$ which has a conic bundle structure
(see Lemma~\ref{conicbdl}).

We conclude the section  with the proof of
Theorem \ref{corindex}. 
\begin{lemma}\label{general}
Let $X$ be a Mori dream space and $D\subset X$ a prime divisor. 
Consider a Mori program for $-D$:
$$
X=X_0\stackrel{\sigma_0}{\dasharrow} X_1 \dasharrow\quad\cdots
\quad
\dasharrow X_{k-1}\stackrel{\sigma_{k-1}}{\dasharrow} X_k.
$$
Let $D_i\subset X_i$ be the transform of $D$, for $i=1,\dotsc,k$, and
set $D_0:=D$, so that $D_i\cdot Q_i>0$ for $i=0,\dotsc,k$. We have
the following.  
\begin{enumerate}[$(1)$]
\item Every $D_i$ is a prime divisor in $X_i$, and
the program ends with an elementary contraction of fiber type
$\ph\colon X_k\to Y$ such that $\NE(\ph)=Q_k$ and
$\ph(D_k)=Y$. 
\item $\ \#\{i\in\{0,\dotsc,k\}\,|\,Q_i\not\subset\N(D_i,X_i)\}=\codim\N(D,X)$. 
\item Set $c_i:=\codim\N(D_i,X_i)$ for $i=0,\dotsc,k$.
For every $i=0,\dotsc,k-1$ we have
$$c_{i+1}=\begin{cases} c_{i}\quad&\text{if $Q_i\subset\N(D_i,X_i)$}\\
c_{i}-1\quad&\text{if $Q_i\not\subset\N(D_i,X_i)$}
\end{cases},\ \text{ and }
c_{k}=\begin{cases} 0\quad&\text{if $Q_k\subset\N(D_k,X_k)$}\\
1\quad&\text{if $Q_k\not\subset\N(D_k,X_k)$.}
\end{cases}$$ 
\item Suppose that $X$ is smooth. Let  
$A_1\subset X_1$ be the indeterminacy locus of $\sigma_0^{-1}$,
and for $i=2,\dotsc,k$, if  $\sigma_{i-1}$ is a divisorial contraction
(respectively, if $\sigma_{i-1}$ is a flip), let
$A_i\subset X_i$ be the union of $\sigma_{i-1}(A_{i-1})$ (respectively, the
transform of $A_{i-1}$) and  
the indeterminacy locus of
$\sigma_{i-1}^{-1}$.

Then for all $i=1,\dotsc,k$
we have $\Sing(X_i)\subseteq A_i\subset D_i$, and
the birational map $X_i\dasharrow X$ is an
isomorphism over $X_i\smallsetminus A_i$.
\end{enumerate}
\end{lemma}
\begin{proof}
Most of the statements are shown in \cite{31} (see in particular
Remarks 2.5 and 2.6, and Lemma 3.6); for
the reader's convenience we give a proof. We have $D_i\cdot Q_i>0$ for
every $i=0,\dotsc,k$, just by the definition of Mori program for
$-D$. 
 
Let $i\in\{0,\dotsc,k-1\}$ be such that $\sigma_i$ is a divisorial
contraction. Then $D_i\neq\Exc(\sigma_i)$ (for otherwise $D_i\cdot
Q_i<0$), hence $D_{i+1}=\sigma_i(D_i)\subset X_{i+1}$ is a prime
divisor. On the other hand 
 $D_i$ intersects every non-trivial fiber of $\sigma_i$ (because
$D_i\cdot Q_i>0$), in particular $D_i\cap \Exc(\sigma_i)\neq\emptyset$
and $D_{i+1}\supset\sigma_i(\Exc(\sigma_i))$. Notice that
$\sigma_i(\Exc(\sigma_i))$ is the indeterminacy locus of
$\sigma_i^{-1}$. 

Consider the push-forward $(\sigma_i)_*\colon
\N(X_i)\to\N(X_{i+1})$. We have $\ker(\sigma_i)_*=\R Q_i$ and
$\N(D_{i+1},X_{i+1})=(\sigma_i)_*(\N(D_i,X_i))$, therefore
$c_{i+1}=c_i$ if $Q_i\subset\N(D_i,X_i)$, and $c_{i+1}=c_i-1$
otherwise. 

\medskip

Now let $i\in\{0,\dotsc,k-1\}$ be such that $\sigma_i$ is a flip,
and consider the standard flip diagram: 
$$\xymatrix{{X_i}\ar[dr]_{\ph_i}\ar@{-->}[rr]^{\sigma_i}&&{X_{i+1}}\ar[dl]^{\ph_i'}\\
&{Y_i}&}$$
where $\ph_i$ is the contraction of $Q_i$, and $\ph_i'$ is the
corresponding
 small elementary contraction of $X_{i+1}$. We have
$D_{i+1}\cdot\NE(\ph_i')<0$, in particular  $\Exc(\ph_i')\subset D_{i+1}$ and
$\NE(\ph_i')\subset\N(D_{i+1},X_{i+1})$. Notice that $\Exc(\ph_i')$ is the
indeterminacy locus of $\sigma_i^{-1}$. 

Moreover $\ph_i(D_i)=\ph_i'(D_{i+1})$, so that 
$$(\ph_i)_*\left(\N(D_i,X_i)\right)=\N(\ph_i(D_i),Y_i) 
=(\ph_i')_*\left(\N(D_{i+1},X_{i+1})\right).$$
Since $\ker(\ph_i')_*\subseteq \N(D_{i+1},X_{i+1})$, we have $c_{i+1}=\codim
\N(\ph_i(D_i),Y_i)$. We deduce again that  $c_{i+1}=c_i$ if
$Q_i\subset\N(D_i,X_i)$, and $c_{i+1}=c_i-1$ otherwise. 

\medskip

In particular the preceding analysis shows that for every
$i=1,\dotsc,k$ the divisor $D_i$ contains the indeterminacy locus of
$\sigma_i^{-1}$, so that $A_i\subset D_i$. By definition, $A_i$ contains the
indeterminacy locus of the birational map
$(\sigma_{i-1}\circ\cdots\circ\sigma_{0})^{-1}\colon X_i\dasharrow X$;
in particular $X_i\smallsetminus A_i$ is isomorphic to an open subset
of $X$, thus it is smooth if $X$ is smooth. This shows $(4)$. 

\medskip

Consider now the prime divisor $D_k\subset X_k$. Clearly $-D_k$ cannot
be nef, therefore the program ends with a fiber type contraction
$\ph\colon X_k\to Y$.  
Since $D_k\cdot Q_k>0$, $D_k$ intersects every fiber of $\ph$, namely
$\ph(D_k)=Y$, and we have $(1)$.

 In particular $\ph_*(\N(D_k,X_k))=\N(Y)$, hence either
$c_k=0$ (\emph{i.e.}\ $\N(D_k,X_k)=\N(X_k)$), or
$c_k=1$ and $Q_k\not\subset\N(D_k,X_k)$. 
Thus we have  $(3)$, which implies directly
$(2)$.
\end{proof}
\begin{lemma}\label{pluto}
Let $X$ be a Fano manifold and $D\subset X$ a prime divisor.
Consider a special Mori program for $-D$:
$$
X=X_0\stackrel{\sigma_0}{\dasharrow} X_1 \dasharrow\quad\cdots
\quad
\dasharrow X_{k-1}\stackrel{\sigma_{k-1}}{\dasharrow} X_k.
$$
Then we
have the following (we keep the notation of Lemma~\ref{general}).
\begin{enumerate}[$(1)$]
\item Let $i\in\{0,\dotsc,k-1\}$ be such that
  $Q_i\not\subset\N(D_i,X_i)$. \\
Then  $Q_{i}$ is of type
$(n-1,n-2)^{sm}$, \emph{i.e.}\ $\sigma_{i}\colon
X_{i}\to X_{i+1}$ is the blow-up of  a smooth
subvariety of codimension $2$, contained in the smooth locus of 
$X_{i+1}$. Moreover $\Exc(\sigma_{i})\cap A_{i}=\emptyset$,
hence $\Exc(\sigma_{i})$ does not intersect
the loci of the birational maps $\sigma_l$ for $l<i$.
\item Set
  $\,s:=\#\{i\in\{0,\dotsc,k-1\}\,|\,Q_i\not\subset\N(D_i,X_i)\}$.
We have two possibilities:\\
 either $\,s=\codim\N(D,X)\,$ and 
$\,\N(D_k,X_k)=\N(X_k)$, \\
or $\,s=\codim\N(D,X)-1,\,$ $\,Q_k\not\subset
\N(D_k,X_k),\,$ and $\,\codim\N(D_k,X_k)=1$.
\item 
Set 
$\,\{i_1,\dotsc,i_s\}:=\left\{i\in\{0,\dotsc,k-1\}\,|\,Q_i\not\subset
\N(D_i,X_i)\right\},\,$ and
let $E_j\subset X$ be the transform of $\Exc(\sigma_{i_j})\subset
 X_{i_j}$ for every $j=1,\dotsc,s$.\\  Then  $E_j$ is a 
smooth $\pr^1$-bundle, with fiber  $f_j\subset E_j$, such that
 $E_j\cdot
f_j=-1$, $D\cdot f_j>0$, and $[f_j]\not\in\N(D,X)$. In particular
$E_j\cap D\neq\emptyset$ and $E_j\neq D$.
\item The prime divisors
$E_1,\dotsc,E_s$ are
pairwise disjoint. 
\end{enumerate}
\end{lemma}
\noindent We call $E_1,\dotsc,E_s$ {\bf the $\pr^1$-bundles determined
  by the special Mori program for $-D$} that we are considering.
These divisors will play a key role
throughout the paper.

Notice that Proposition \ref{MP2} is a straightforward consequence of
Proposition \ref{tomm} and of Lemma~\ref{pluto}, more precisely
of \ref{pluto}$(3)$ and \ref{pluto}$(4)$. 
\begin{proof}
Statement  $(1)$ follows from
\cite[Lemma~3.9]{31}.

By \ref{general}$(2)$ we have 
$$s=\begin{cases}\codim\N(D,X)\quad&\text{if $Q_k\subset \N(D_k,X_k)$,}\\
\codim\N(D,X)-1\quad&\text{if $Q_k\not\subset\N(D_k,X_k)$.}
\end{cases}$$
 Together with \ref{general}$(3)$ this yields
 $(2)$.

Let $j\in\{1,\dotsc,s\}$. By (1) we have $E_j\cong\Exc(\sigma_{i_j})$,
thus  $E_j$ is a smooth $\pr^1$-bundle with
 $E_j\cdot f_j=-1$, where $f_j\subset E_j$ is a fiber, and
 $D\cdot f_j>0$ because $D_{i_j}\cdot Q_{i_j}>0$ in $X_{i_j}$. 
In particular
$E_j\cap D\neq\emptyset$ and $E_j\neq D$.
Moreover
$[f_j]\subset\N(D,X)$ would yield $Q_{i_j}\subset\N(D_{i_j},X_{i_j})$,
 which is excluded by definition. Therefore we have $(3)$.

Finally $E_1,\dotsc,E_s$ are pairwise disjoint, because for $j=1,\dotsc,s$
the divisor  $\Exc(\sigma_{i_j})$ does not intersect
the loci of the previous birational maps.
\end{proof}
Here is a more detailed description of the case where
$s=\codim\N(D,X)-1$ in Lemma~\ref{pluto}. 
\begin{lemma}[Conic bundle case]\label{conicbdl}
Let $X$ be a Fano manifold and $D\subset X$ a prime
divisor. Consider a special Mori program for $-D$; we
 keep the same notation as in
Lemmas~\ref{general} and~\ref{pluto}.
Set $c:=\codim\N(D,X)$,
 $\sigma:=\sigma_{k-1}\circ\cdots\circ\sigma_0\colon X\dasharrow
X_k$, and $\psi:=\ph\circ\sigma\colon
X\dasharrow Y$. 
$$\xymatrix{
{X=X_0}\ar@/^1pc/@{-->}[rrrr]^{\sigma}\ar@/_1pc/@{-->}[drrrr]_{\psi}
\ar@{-->}[r]_{\,\quad \sigma_0}& 
{X_1}\ar@{-->}[r]&
{\cdots}\ar@{-->}[r] &
{X_{k-1}}\ar@{-->}[r]_{\sigma_{k-1}}& 
{X_k}\ar[d]^{\ph}\\
&&&&Y 
}$$
We assume that $Q_k\not\subset\N(D_k,X_k)$, equivalently that $s=c-1$
(see \ref{pluto}(2)). 
Then we have the following.
\begin{enumerate}[$(1)$]
\item Every fiber of $\ph$ has dimension $1$, $\dim Y=n-1$, and
$\ph$ is finite on $D_k$.
\item  Let
$j\in\{1,\dotsc,c-1\}$ and consider
  $\sigma_{i_j}(\Exc(\sigma_{i_j}))\subset X_{i_j+1}$. 
For every $m=i_j+1,\dotsc,k-1\,$ the set
 $\Lo(Q_m)\subset X_m$ is disjoint from the image of
  $\sigma_{i_j}(\Exc(\sigma_{i_j}))$ in $X_m$, so that the birational map
$X_{i_j+1}\dasharrow X_k$ is an isomorphism on
$\sigma_{i_j}(\Exc(\sigma_{i_j}))$, and $\sigma$ is regular on $E_j\subset
X$. 
\item
There exist open subsets $U\subseteq X$ and $V\subseteq Y$, with
$E_1,\dotsc,E_{c-1}\subset U$, such
that  $V$ and $\ph^{-1}(V)$ are smooth,
$\ph_{|\ph^{-1}(V)}\colon\ph^{-1}(V)\to V$ and  
$\psi\colon U\to V$ are conic bundles, and $\sigma_{|U}$ is the
blow-up of pairwise disjoint smooth subvarieties
$T_1,\dotsc,T_{c-1}\subset \ph^{-1}(V)$, of dimension $n-2$, with
exceptional divisors $E_1,\dotsc,E_{c-1}$. 
$$
\xymatrix{U \ar@/^1pc/[rr]^{\psi}\ar[r]_(.4){\sigma_{|U}}
  &{\ph^{-1}(V)}\ar[r]_(.6){\ph}& V}$$
In particular we have $\Lo(Q_m)\subseteq X_m\smallsetminus
(\sigma_{m-1}\circ\cdots \circ\sigma_0)(U)$
 for every $m\in\{0,\dotsc,k-1\}\smallsetminus
  \{i_1,\dotsc,i_{c-1}\}$.
\item
Set $Z_j:=\psi(E_j)\subset V$ for every $j\in\{1,\dotsc,c-1\}$.
Then  $Z_1,\dotsc,Z_{c-1}\subset Y$ are pairwise disjoint smooth prime divisors,
 and 
$\psi^*(Z_j)=E_j+\widehat{E}_j$, where   
$\widehat{E}_j\subset U$ is a smooth $\pr^1$-bundle with fiber
$\widehat{f}_j\subset\widehat{E}_j$, 
$f_j+ \widehat{f}_j$ is numerically equivalent to a general fiber of
$\psi$, and
$$\widehat{E}_j\cdot \widehat{f}_j=-1,\quad
E_j\cdot \widehat{f}_j=\widehat{E}_j\cdot {f}_j=1,\ \text{ and }\
[\widehat{f}_j]\not\in\N(E_j,X),$$
 for every $j\in\{1,\dotsc,c-1\}$.
In particular the divisors
$D,E_1,\dotsc,E_{c-1},\widehat{E}_1,\dotsc,\widehat{E}_{c-1}$ are all
distinct, and 
$E_1\cup\widehat{E}_1,\dotsc,E_{c-1}\cup\widehat{E}_{c-1}$ are pairwise disjoint.
\end{enumerate}
\end{lemma}
We refer the reader to \cite[p.~1478-1479]{fano} for an 
 explicit
description of the rational conic bundle $\psi$ in the toric case.
\begin{proof}[Proof of Lemma~\ref{conicbdl}]
Let $F\subset X_k$ be a fiber of $\ph$. Then $F\cap D_k\neq\emptyset$
because $D_k\cdot Q_k>0$; on the other hand $\dim(F\cap D_k)=0$,
because if there exists a curve $C\subset F\cap D_k$, then $[C]\in
Q_k$ and $[C]\in\N(D_k,X_k)$, thus $Q_k\subset\N(D_k,X_k)$ against
our assumptions.  Hence
every fiber of $\ph$ has dimension
$1$, $\dim Y=n-1$, and we have $(1)$. 

Recall from \ref{general}$(4)$ that $\Sing(X_k)\subseteq A_k$, and
notice that $\codim A_k\geq 2$, therefore
 $A_k$ cannot dominate $Y$. 
Restricting  $\ph$ we
get a contraction  
$X_k\smallsetminus\ph^{-1}(\ph(A_k))\to Y\smallsetminus
\ph(A_k)$  of a smooth variety, with $-K_{X_k}$ relatively
ample (because $-K_{X_k}\cdot Q_k>0$), and one-dimensional fibers.  
 We conclude that  $Y\smallsetminus
\ph(A_k)$ is smooth and that
$\ph_{|X_k\smallsetminus\ph^{-1}(\ph(A_k))}$ is a conic bundle (see
\cite[Theorem 4.1(2)]{AWaview}). 

By \ref{general}$(4)$, $\sigma\colon X\dasharrow X_k$ is an isomorphism over
$X_k\smallsetminus A_k$. If
$U_1:=\sigma^{-1}(X_k\smallsetminus\ph^{-1}(\ph(A_k)))$, then
$\psi\colon U_1\to Y\smallsetminus\ph(A_k)$
is again a conic bundle; in particular it is flat, and induces an
injective morphism $\iota\colon   Y\smallsetminus\ph(A_k)\to\Hilb(X)$.
Let $H\subset\Hilb(X)$ be the closure of the image of $\iota$, and
$\mathcal{C}\subset H\times X$ the restriction of the universal family over
$\Hilb(X)$. We get a diagram:
$$\xymatrix{{\mathcal{C}}\ar[d]_{\pi}\ar[r]^e&
X\ar@{-->}[r]^{\sigma}\ar@{-->}[dr]^{\psi} & {X_k}\ar[d]^{\ph}\\
H&&Y\ar@{-->}[ll]^{\iota} }$$
where 
$\pi\colon\mathcal{C}\to H$ and $e\colon\mathcal{C}\to X$ are the
projections, and $\iota$ is birational.  
We want to compare the degenerations 
in $X$ and in $X_k$ 
of the general fibers the conic bundle $\psi_{|U_1}$.

Fix $j\in\{1,\dotsc,c-1\}$, and recall from \ref{pluto}$(1)$ that
$\Exc(\sigma_{i_j})\cap A_{i_j}=\emptyset$, so that the birational map
$X\dasharrow X_{i_j}$ is an isomorphism over
$\Exc(\sigma_{i_j})$. In $X_{i_j+1}$ we have 
$$A_{i_j+1}=\sigma_{i_j}\left(\Exc(\sigma_{i_j})\cup A_{i_j}\right),$$
hence $\sigma_{i_j}(\Exc(\sigma_{i_j}))$ is a connected component of $A_{i_j+1}$.

Let $x\in \sigma_{i_j}(\Exc(\sigma_{i_j}))\subset X_{i_j+1}$ and let
$l\subseteq E_j\subset X$ be the transform of the fiber of
$\sigma_{i_j}$ over~$x$.

Let $B_0\subseteq H$ be a general irreducible curve
 which intersects $\pi(e^{-1}(l))$. Since $\pi$ is equidimensional
and the general fiber of $\pi$ over $B_0$ is $\pr^1$, the inverse
image $\pi^{-1}(B_0)\subseteq \mathcal{C}$ is irreducible. Set
$S:=e(\pi^{-1}(B_0))\subseteq X$, then $S\cap l\neq\emptyset$ by
construction.

Consider the normalizations
 $B\to B_0$ and $\mathcal{C}_B\to \pi^{-1}(B_0)$ of $B_0$ and
$\pi^{-1}(B_0)$ respectively;  
we have induced morphisms $e_B\colon \mathcal{C}_B\to S$ and 
$\pi_B\colon \mathcal{C}_B\to B$. 
$$\xymatrix{
{\mathcal{C}_B}\ar[d]^{\pi_B}\ar[r] &
{\pi^{-1}(B_0)\subseteq\mathcal{C}}\ar[d]^{\pi}\ar[r]^(.4){e} &
{X\supseteq S:=e(\pi^{-1}(B_0))}\\
B\ar[r]&{B_0\subseteq H}&
}$$
Because $B_0$ is general,
 $B_0\cap\dom(\iota^{-1}) 
\neq\emptyset$, and $\iota^{-1}$ induces a morphism
$\eta\colon B\to Y$. Set 
$B_1:=\eta(B)\subset Y$.

Again, since $\ph$ is equidimensional and the general fiber of $\ph$
over $B_1$ is $\pr^1$, the inverse
image $\ph^{-1}(B_1)\subset X_k$ is irreducible; call $S_k$ this
surface, which is just the 
transform of $S\subset X$ under $\sigma$. 

Recall that $\ph$ is finite on $D_k$ by (1), and $A_k\subset D_k$ by 
\ref{general}$(4)$,
hence no component of a 
fiber of $\ph$ can be
contained in $A_k$. On the other hand, by the generality of $B_0$, the
general fiber of $\ph_{|S_k}$ does not intersect $A_k$.
 Therefore $S_k$ can intersect $A_k$ at most in
a finite number of points.

 Consider now $\sigma_{S}:=\sigma_{|S}\colon
S\dasharrow S_k$. Then $\sigma_S$ is an isomorphism over
$S_k\smallsetminus (S_k\cap A_k)$ and $\dim (S_k\cap A_k)\leq 0$, 
hence by Zariski's main theorem
$\xi:=\sigma_S\circ e_B\colon \mathcal{C}_B\to S_k$ is a
morphism.
$$\xymatrix{{\mathcal{C}_B}\ar[d]_{\pi_B}\ar[r]_(0.4){e_B}\ar@/^1pc/[rr]^{\xi} 
& {S\subset X}\ar@{-->}[r]_{\sigma_S}&{S_k\subset X_k}\ar[d]^{\ph} \\
B\ar[rr]^{\eta}&&{B_1\subset Y}}$$

Let $y\in B$ be such that
$C:=e_B(\pi_B^{-1}(y))\subset S$
intersects $l$; in particular $C\cap E_j\neq\emptyset$, because
$l\subseteq E_j$. Since $C$ is numerically equivalent in $X$ to a
general fiber of $\psi$, we have $-K_X\cdot C=2$ and $E_j\cdot C=0$;
in particular $C$
has at most two irreducible components, because $-K_X$ is ample.

Set $r:=\ph^{-1}(\eta(y))$. Since $r$ is numerically equivalent 
in $X_k$ to a
general fiber of $\ph$, we have $-K_{X_k}\cdot r=2$. 
Recall that no irreducible component
of $r$ can be contained in $A_k$; on the other hand, 
$r$ must intersect $A_k$, otherwise $\sigma_S$ would be an isomorphism
over $r$, $C=\sigma_S^{-1}(r)$, and $C\cap E_j=\emptyset$, a
contradiction.  

Let us show that $r$ is an integral fiber of $\ph$.
Indeed 
let  $C_1$ be an irreducible component of $r$. If $C_1\cap
A_k=\emptyset$, then $C_1$ is contained in the smooth locus of $X_k$
and $-K_{X_k}\cdot C_1\geq 1$. If instead $C_1\cap
A_k\neq \emptyset$, then \cite[Lemma~3.8]{31} gives $-K_{X_k}\cdot
C_1> 1$. Since $-K_{X_k}\cdot r=2$ and
$r$ must intersect $A_k$, it must
be irreducible and reduced.

For every $i\in\{0,\dotsc,k-1\}$ let
$\widetilde{r}_i\subset X_i$ be 
the transform of $r\subset X_k$ (where $X_0=X$). Again by
\cite[Lemma~3.8]{31} we get $-K_X\cdot \widetilde{r}_0<-K_{X_k}\cdot
r=2$, hence $-K_X\cdot \widetilde{r}_0=1$. 

Notice that  $\xi(\pi_B^{-1}(y))\subset S_k$ is contained in $r$; on
the other hand $\xi$ cannot contract to a point a fiber of $\pi_B$,
hence $\xi(\pi_B^{-1}(y))=r$. Then 
 $\widetilde{r}_0\subseteq C$, because $C=e_B(\pi_B^{-1}(y))$, and we get
 $C=\widetilde{r}_0\cup
C'$, where $C'\subset X$ is an irreducible curve (and possibly
$C'=\widetilde{r}_0$ if $C$ is non-reduced). 

Since $r\not\subset A_k$, we have $\widetilde{r}_0\not\subset E_j$; in
particular $E_j\cdot\widetilde{r}_0\geq 0$. 
If  $E_j\cdot\widetilde{r}_0
=0$, then also $E_j\cdot C'=0$ and $C\subset E_j$, which is
impossible. Hence 
 $E_j\cdot\widetilde{r}_0>0$, and since $E_j\cdot C=0$, we have
$E_j\cdot C'<0$ and $C'\neq\widetilde{r}_0$.

Consider now the blow-up $\sigma_{i_j}\colon X_{i_j}\to X_{i_j+1}$. We
have $\Exc(\sigma_{i_j})\cdot
\widetilde{r}_{i_j}=E_j\cdot\widetilde{r}_0
\geq 1$, hence  using the projection formula
we get $-K_{X_{i_j+1}}\cdot \widetilde{r}_{i_j+1}\geq
-K_{X_{i_j}}\cdot \widetilde{r}_{i_j}+1$. On the other hand 
\cite[Lemma~3.8]{31} gives
$$1=-K_X\cdot \widetilde{r}_0 \leq -K_{X_{i_j}}\cdot
\widetilde{r}_{i_j}\quad\text{and}\quad
-K_{X_{i_j+1}}\cdot \widetilde{r}_{i_j+1}\leq
-K_{X_k}\cdot r=2.$$
We conclude that $\Exc(\sigma_{i_j})\cdot
\widetilde{r}_{i_j}=1$, 
$-K_X\cdot \widetilde{r}_0 =-K_{X_{i_j}}\cdot
\widetilde{r}_{i_j}$, and $-K_{X_{i_j+1}}\cdot \widetilde{r}_{i_j+1}=
-K_{X_k}\cdot r$, and again by \cite[Lemma~3.8]{31}
this implies that:
\stepcounter{thm}
\begin{equation}\label{auto}
 \text{for every $m\in\{0,\dotsc,k-1\},m\neq i_j$,
$\Lo(Q_m)$ is disjoint from $\widetilde{r}_m$.} 
\end{equation}

We show that $C'=l$ (recall that $l\subset X$ is the transform of 
 $\sigma_{i_j}^{-1}(x)\subset X_{i_j}$).
 Since $C'$ intersects $\widetilde{r}_0$ (because
 $C=\widetilde{r}_0\cup C'$ is connected), and
$\widetilde{r}_0\cap\Lo(Q_0)=\emptyset$ by \eqref{auto}, we see that $C'$ is not
contained in $\Lo(Q_0)$. Iterating this reasoning 
for every $\sigma_m$ with
$m\in\{0,\dotsc,i_j-1\}$, we see that $C'$ intersects the open subset
 where the birational map
$X\dasharrow X_{i_j}$ is an isomorphism; let $\widetilde{C}'\subset X_{i_j}$
be its transform. 

If $\sigma_{i_j}(\widetilde{C}')$ were a curve, then by the same
reasoning it could
not be contained in $\Lo(Q_m)$ for any $m=i_j+1,\dotsc,k-1$, and
in the end we would get a curve $\widetilde{C}'_{k}\subset X_k$, 
distinct from
$r$, which should belong to $\xi(\pi_B^{-1}(y))$, which is impossible.
Thus $\widetilde{C}'$ must be a fiber of $\sigma_{i_j}$. On the
other hand $\Exc(\sigma_{i_j})\cdot
\widetilde{r}_{i_j}=1$, thus $\widetilde{r}_{i_j}$ intersects a unique
fiber of  $\sigma_{i_j}$, and $C'=l$.

In particular this yields that $x\in\widetilde{r}_{i_j+1}\cap
\sigma_{i_j}(\Exc(\sigma_{i_j}))$. Since $x\in
\sigma_{i_j}(\Exc(\sigma_{i_j}))$
 was arbitrary, \eqref{auto} implies statement $(2)$.

\medskip

Let $T_j\subset X_k$ be the image of $\sigma_{i_j}(\Exc(\sigma_{i_j}))\subset
  X_{i_j+1}$. By $(2)$ the birational map
  $X_{i_j+1}\dasharrow X_k$ yields an isomorphism between
  $\sigma_{i_j}(\Exc(\sigma_{i_j}))$ and $T_j$, hence $T_j$ is smooth
  of dimension $n-2$, and is contained in the smooth locus of $X_k$.
 Since
  $\sigma_{i_j}(\Exc(\sigma_{i_j}))$ is a connected component of
  $A_{i_j+1}$, we deduce that $T_j$ is a connected component of $A_k$, and
  $A_k\smallsetminus T_j$ is closed in $X_k$.

By \eqref{auto} the birational map 
 $X_{i_j+1}\dasharrow X_k$ yields also an isomorphism between
$\widetilde{r}_{i_j+1}$ and $r$, and  $r\cap (A_k\smallsetminus
T_j)=\emptyset$. 

Consider the point $x'\in T_j$ corresponding to  $x\in
\sigma_{i_j}(\Exc(\sigma_{i_j}))$. 
Then $x'\in r\cap T_j$ because $x\in\widetilde{r}_{i_j+1}$, \emph{i.e.}\
 $r$ is the fiber of $\ph$ through $x'\in T_j$. Again since  $x$
 was arbitrary in
$\sigma_{i_j}(\Exc(\sigma_{i_j}))$, from  $r\cap (A_k\smallsetminus
T_j)=\emptyset$ we deduce that
 $\ph^{-1}(\ph(T_j))\cap (A_k\smallsetminus T_j)=\emptyset$, and hence
that $\ph(T_j)\cap\ph(A_k\smallsetminus T_j)=\emptyset$ in $Y$.

\medskip

Summing up, we have shown that $T_1,\dotsc,T_{c-1}$ are connected
components of $A_k$ (so that $A_k\smallsetminus (T_1\cup\cdots\cup
T_{c-1})$ is closed in $X_k$), and
 the images
  $\ph(T_1),\dotsc,\ph(T_{c-1}),\ph(A_k\smallsetminus(T_1\cup\cdots 
\cup T_{c-1}))$ are 
pairwise disjoint in $Y$.

Now set
\stepcounter{thm}
\begin{equation}\label{V}
V:=Y\smallsetminus \ph(A_k\smallsetminus(T_1\cup\cdots
\cup T_{c-1})).\end{equation}
Then $V$ is open in $Y$,
$\ph^{-1}(V)\subseteq\sigma(\dom(\sigma))$, 
and
$T_1\cup\cdots
\cup T_{c-1}\subset\ph^{-1}(V)$. Set
 $U:=\sigma^{-1}(\ph^{-1}(V))\subseteq X$. By definition,
$\ph^{-1}(V)\cap (A_k\smallsetminus(T_1\cup\cdots
\cup T_{c-1}))=\emptyset$; this means that
for every $m\in\{0,\dotsc,k-1\}\smallsetminus
  \{i_1,\dotsc,i_{c-1}\}$, $\Lo(Q_m)$ is disjoint from the image of
  $U$ in $X_m$.

We have $E_1,\dotsc,E_{c-1}\subset
U$, because $E_j=\sigma^{-1}(T_j)$, and
 $\psi\colon U\to V$ is regular and proper. More precisely, every
fiber of $\psi$ over  $V$ is one-dimensional, and as before 
\cite[Theorem 4.1(2)]{AWaview} shows that this is
 a conic bundle and that $V$ is smooth. We have a factorization 
$$
\xymatrix{U \ar@/^1pc/[rr]^{\psi}\ar[r]_(.4){\sigma_{|U}}
  &{\ph^{-1}(V)}\ar[r]_(.6){\ph}& V}$$
and $\sigma_{|U}$ is just the blow-up of $T_1\cup\cdots 
\cup T_{c-1}$, so we get $(3)$. For every $j\in\{1,\dotsc,c-1\}$ we have
$Z_j=\psi(E_j)=\ph(T_j)$, so $Z_1,\dotsc,Z_{c-1}$ are pairwise disjoint.
Now let $\widehat{E}_j\subset U$ be the transform of
$\ph^{-1}(Z_j)$. Then $\psi^{-1}(Z_j)=E_j\cup
\widehat{E}_j$, and the rest of statement $(4)$ 
follows from standard arguments on
conic bundles.
Just notice that if for some $j\in\{1,\dotsc,c-1\}$
we have $[\widehat{f}_j]\in\N(E_j,X)$, then
$[\sigma(\widehat{f}_j)]\in\N(T_j,X_k)\subseteq\N(A_k,X_k)\subseteq
\N(D_k,X_k)$, which is impossible because $\sigma(\widehat{f}_j)$ is a
fiber of $\ph$ and $\NE(\ph)\not\subset\N(D_k,X_k)$ by assumption.
\end{proof}
\begin{proof}[Proof of  Theorem \ref{corindex}]
If $\N(D,X)=\N(X)$ for every prime 
divisor $D\subset X$, then we have $(ii)$  
(just notice that
if $D_1,D_2\subset X$ are two disjoint divisors, then
$\N(D_1,X)\subseteq D_2^{\perp}\subsetneq\N(X)$, see Remark \ref{elem}).

Suppose now that there exists a prime divisor $D\subset X$ with
$\codim\N(D,X)>0$, and consider a special Mori program for $-D$ (which
exists by Proposition \ref{tomm}).
Let $E_1,\dotsc,E_s\subset X$ be the $\pr^1$-bundles determined by the
Mori program.

If $s\geq 1$, by \ref{pluto}$(3)$ we have $-K_X\cdot f_1=1$, where
 $f_1\subset E_1$ is a fiber of the $\pr^1$-bundle; this is impossible
because $\iota_X>1$. 

 Therefore $s=0$, and \ref{pluto}$(2)$ yields 
that $\codim\N(D,X)=1$ and $Q_k\not\subset\N(D_k,X_k)$, so that
 Lemma~\ref{conicbdl} applies.

We show that $k=0$ and $X=X_k$. Indeed if
not, we have $A_k\neq\emptyset$ in $X_k$ (see \ref{general}$(4)$). 
Take $r$ a fiber of $\ph$
intersecting $A_k$. Then, using \cite[Lemma~3.8]{31} 
as in the proof of Lemma~\ref{conicbdl}, we
see that $r$ is integral, and that the transform
$\widetilde{r}\subset X$ of $r$ has anticanonical degree $1$ in $X$, 
 a contradiction.

Thus $X=X_k$ and we get a conic bundle
$\ph\colon X\to Y$, which is finite on $D$. Since  $X$ contains no
curves of anticanonical degree $1$, 
$\ph$ must be a smooth fibration in $\pr^1$. 
Then $Y$ is Fano by \cite[Proposition 4.3]{wisn}, and
finally we have $\iota_Y\geq \iota_X=2$ by
\cite[Lemme~2.5]{mukai}.
\end{proof}
\section{Divisors with minimal Picard number}\label{Divisors}
Let $X$ be a Fano manifold, and consider 
$$c_X:=\max\{\codim\N(D,X)\,|\,D\text{ is a prime divisor in $X$}\}.$$
We always have $0\leq c_X\leq \rho_X-1$.
If $S$ is a Del Pezzo surface, then $c_S=\rho_S-1\in\{0,\dotsc,8\}$.
\begin{example}\label{max}
Consider a Fano manifold $X=S\times T$, where $S$ is a Del Pezzo
surface. Then $c_X=\max\{\rho_S-1,c_T\}$. More precisely, for any
prime divisor $D\subset X$, we have three possibilities:
\begin{enumerate}[$\bullet$]
\item $D=C\times T$ where $C\subset S$ is a curve,
and $\codim\N(D,X)=\rho_S-1$;
\item
$D=S\times D_T$ where $D_T\subset T$ is a divisor,
and $\codim\N(D,X)=\codim\N(D_T,T)\leq c_T$;
\item
$D$ dominates both $S$ and $T$ under the projections, and
$\codim\N(D,X)\leq\rho_S-1$. 
\end{enumerate}
Indeed suppose that
 $D\subset X$ is a
prime divisor with $\codim\N(D,X)>\rho_{S}-1$. Then
$\dim\N(D,X)<\rho_{T}+1$, so that $D$ cannot dominate $T$ under
the projection, and $D=S\times D_T$.
\end{example}
\begin{example}
If $X$ is a Fano manifold with pseudo-index $\iota_X\geq 3$ (for
instance $X=\pr^{n_1}\times\cdots\times\pr^{n_r}$ with $n_i\geq 2$ for
all $i=1,\dotsc,r$), then
$c_X=0$ by Theorem \ref{corindex}.
\end{example}
We are going to use the results of section~\ref{run} to
prove the following. 
\begin{thm}\label{gen}
For any Fano manifold $X$ we have $c_X\leq 8$. Moreover:
\begin{enumerate}[$\bullet$]
\item if $c_X\geq 4$ then 
 $X\cong S\times T$ where $S$ is a Del Pezzo surface, $\rho_S=c_X+1$,
 and $c_T\leq c_X$;
\item if $c_X=3$ then there exists a flat, quasi-elementary
contraction $X\to T$ where $T$ is an $(n-2)$-dimensional Fano
manifold, $\rho_X-\rho_T=4$, and $c_T\leq 3$.
\end{enumerate}
\end{thm}
\noindent A {contraction} $\ph$
is \emph{quasi-elementary}
 if $\ker\ph_*$ is generated by the numerical classes of the curves 
 contained in a
 general fiber of $\ph$; we refer the reader to
 \cite{fanos} for properties of quasi-elementary contractions.
In particular, in the case where $c_X=3$ in Theorem \ref{gen},
 the general fiber of the
contraction $X\to T$ is a Del Pezzo surface $S$ with $\rho_S\geq 4$.
\begin{example}[Codimension $3$]\label{cod3}
Let $n\geq 3$ and $Z=\pr_{\pr^{n-2}}(\mathcal{O}^{\oplus
  2}\oplus\mathcal{O}(1))$. Then $Z$ is a toric Fano manifold with
$\rho_Z=2$, and the $\pr^2$-bundle $Z\to \pr^{n-2}$ has three pairwise
disjoint sections $T_1,T_2,T_3\subset Z$ which are closed under the
  torus action. Let $X\to Z$ be 
the blow-up of $T_1,T_2,T_3$. Then $X$ is Fano with $\rho_X=5$, and it
has a smooth morphism $X\to\pr^{n-2}$ such that every fiber is the Del Pezzo
surface $S$ with $\rho_S=4$. If
$E\subset X$ is one of the exceptional divisors of the blow-up, one
easily checks that $\rho_X-\rho_E=\codim\N(E,X)=3$, hence $c_X\geq 3$.
However $X$ is not a product, thus
$c_X= 3$ by Theorem \ref{gen}. 
\end{example}
\begin{parg}
 \label{primooutline}
The proof of Theorem \ref{gen} will take all the rest
of section~3; we will proceed in several steps.
Section \ref{prel} gathers some preliminary remarks and lemmas.
In section~\ref{geq4} we treat the case $c_X\geq 4$, and we show
that $X\cong S\times T$, where $S$ is a Del Pezzo
surface with $\rho_S=c_X+1$, and $T$ a Fano manifold with $c_T\leq
c_X$ (see  
Proposition \ref{primameta}, and \ref{secondooutline} for an outline
of its proof).  
In particular this implies that $c_X\leq 8$,
because $\rho_S\leq 9$. 

The case $c_X=3$ is more delicate, as we have to treat separately the two
following  cases:\refstepcounter{thm}\label{caso}
\begin{enumerate}[(\thethm.a)]
\item  for every prime divisor $D\subset
X$ with $\codim\N(D,X)=3$, and for every
special Mori program for $-D$,
 we have $\N(D_k,X_k)=\N(X_k)$ (notation as in
Lemma~\ref{general});
\item there exist a prime divisor $D\subset
X$ with $\codim\N(D,X)=3$, and a
special Mori program for $-D$, such that $\N(D_k,X_k)\subsetneq\N(X_k)$.
\end{enumerate}
The first case 
(\ref{caso}.a) is treated together with the case $c_X\geq 4$, in
section~\ref{geq4}. In the end we reach a contradiction, hence \emph{a
  posteriori} we conclude that (\ref{caso}.a)  never happens (see
Corollary \ref{asilo}).
 The
second  case (\ref{caso}.b)
is treated in section~\ref{sectiontre}, where we show the existence
of a  flat, quasi-elementary
contraction $X\to T$, where $T$ is an $(n-2)$-dimensional Fano
manifold, $\rho_X-\rho_T=4$, and $c_T\leq 3$ (see Proposition \ref{secondameta},
 and \ref{terzooutline} for an outline of its proof).
\end{parg}
\subsection{Preliminary results}\label{prel}
In this section we collect some remarks and lemmas which will be used
in the proof of Theorem \ref{gen}.
\begin{remark2}\label{divisors}
Let $X$ be a projective manifold, $\ph\colon X\to Y$ a contraction such that 
$-K_X$ is $\ph$-ample and $\dim Y>0$, and $D$ a divisor in $X$ such that
$\ker\ph_*\subseteq D^{\perp}$. Then we have the following:
\begin{enumerate}[$(1)$]
\item $\dim Y=1+\dim \ph(\Supp D)$ and $D=\ph^*(D_Y)$, $D_Y$ a Cartier
divisor in $Y$;
\item
if $D$ is a prime divisor, then $\ph(D)$ is a prime Cartier divisor,
and $D=\ph^*(\ph(D))$;
\item if $D$ is a smooth prime divisor,  let
$\ph(D)^{\nu}\to
\ph(D)$ be the normalization. Then the morphism
$\ph_D\colon D\to \ph(D)^{\nu}$ induced by $\ph_{|D}$ is a contraction, and
$-K_D$ is $\ph_D$-ample;
\item
if $D$ is a smooth prime divisor and $Y$ is smooth, then $\ph(D)$ is a
smooth prime divisor.  
\end{enumerate}
\end{remark2}
\begin{proof}
By \cite[Theorem 3.7(4)]{kollarmori} there exists a Cartier divisor $D_Y$
on $Y$ such that $D=\ph^*(D_Y)$. Then $\Supp D_Y=\ph(\Supp D)$, so we
have $(1)$.

If $D$ is a prime divisor, then $D_Y$ is a prime divisor supported on $\ph(D)$,
namely $D_Y=\ph(D)$, and we have $(2)$.
 
For $(3)$, $\ph_D$ is surjective with connected fibers onto a normal
projective variety, 
hence a contraction.
Let $i\colon D\hookrightarrow X$ be the inclusion and take
$\gamma\in\overline{\NE}(D)\cap\ker(\ph_{D})_*$ with $\gamma\neq 0$. 
The restriction $(-K_X)_{|D}$  is $\ph_D$-ample, hence 
$(-K_X)_{|D}\cdot\gamma>0$. Moreover
$i_*(\gamma)\in\ker \ph_*$, so that
$$-K_D\cdot\gamma=-(K_X+D)\cdot i_*(\gamma)=-K_X\cdot
i_*(\gamma)>0,$$
and $-K_D$ is $\ph_{D}$-ample.

For $(4)$, let $y\in \ph(D)$ and let
$f\in\mathcal{O}_{Y,y}$ be a local equation for $\ph(D)$. Then
$\ph^*(f)$ is a local equation for $D$ near the fiber over $y$.
Since $D$ is smooth, the differential $d_x(\ph^*(f))$ is
non-zero, where $x\in \ph^{-1}(y)$. Then $d_yf$ is
non-zero, hence $\ph(D)$ is smooth at $y$.
\end{proof}
\begin{remark2}\label{elem}
Let $X$ be a projective manifold, $Z\subset X$ a closed subset,
 and $D\subset X$ a prime divisor.
If $Z\cap D=\emptyset$, then $D\cdot C=0$ for every curve $C\subset
Z$, hence $\N(Z,X)\subseteq D^{\perp}$.
\end{remark2}
\begin{remark2}\label{elem2}
Let $X$ be a projective manifold,
$E\subset X$ a smooth prime divisor which is a $\pr^1$-bundle with
fiber $f\subset E$, and $D\subset X$ a prime divisor with $D\cdot
f>0$. Then the following holds:
\begin{enumerate}[$(1)$]
\item $\dim\N(D\cap E,X)\geq\dim\N(E,X)-1\,$ and
  $\,\N(E,X)=\R[f]+\N(D\cap E,X)$;
\item
either $[f]\in\N(D\cap E,X)$ and $\N(D\cap E,X)=\N(E,X)$, 
or $[f]\not\in\N(D\cap E,X)$ and $\N(D\cap E,X)$ has codimension
  $1$ in $\N(E,X)$; 
\item
for every irreducible curve $C\subset E$ we have
$C\equiv \lambda
f + \mu C'$,
 where $C'$ is an irreducible 
curve contained in $D\cap E$, $\lambda,\mu\in\R$, and
$\mu\geq 0$.
\end{enumerate}
\end{remark2}
\begin{proof}
Let $\pi\colon E\to F$ be the $\pr^1$-bundle structure on $E$,
and consider the push-forward $\pi_*\colon \N(E)\to\N(F)$. This is a
surjective linear map with kernel $\R[f]_E$.

Since
$D\cdot f>0$, we have $\pi(D\cap E)=F$, thus $\pi_*(\N(D\cap
E,E))=\N(F)$. Therefore $\N(E)=\R[f]_E+\N(D\cap E,E)$, and applying $i_*$
(where $i\colon E\hookrightarrow X$ is the inclusion) we get $(1)$ and
$(2)$. Statement $(3)$ follows from
\cite[Lemma 3.2 and Remark 3.3]{occhetta}. 
\end{proof}  
\begin{remark2}\label{extremal}
Let $X$ be a Fano manifold and $D,E\subset X$ prime divisors with
$$\N(D\cap E,X)\subseteq E^{\perp}.$$ 
Suppose that $E$ is a smooth $\pr^1$-bundle with fiber $f\subset E$,
such that $E\cdot f=-1$ and $D\cdot f>0$. 

Then the half-line $\R_{\geq 0}[f]\subset\NE(X)$ is 
an extremal ray of type 
  $(n-1,n-2)^{sm}$, with contraction $\ph\colon X\to Y$ where
$E=\Exc(\ph)$ and
  $Y$ is Fano. 
\end{remark2}
\begin{proof}
Notice first of all that
 $(-K_X+E)\cdot f=0$. 

Let $C\subset X$ be an irreducible curve.
If $C\not\subset E$, then  $(-K_X+E)\cdot
C >0$. If $C\subseteq D\cap E$, then $E\cdot C=0$, and again  $(-K_X+E)\cdot
C >0$. 

Assume now that $C\subseteq E$.
By \ref{elem2}$(3)$ 
 we have $C\equiv \lambda
f + \mu C'$, where $C'$ is a curve contained in $D\cap E$, $\lambda,\mu\in\R$,
 and $\mu\geq 0$. Thus
$$(-K_X+E)\cdot C=\mu(-K_X+E)\cdot C'\geq 0,$$ 
and  $(-K_X+E)\cdot C=0$ if and only if
$\mu=0$, if and only if $[C]\in \R_{\geq 0}[f]$. Therefore
$-K_X+E$ is nef, and $(-K_X+E)^{\perp}\cap\NE(X)=\R_{\geq 0}[f]$ is an
extremal ray.

Let $\ph\colon X\to Y$ be the contraction of $\R_{\geq 0}[f]$; clearly
$\Exc(\ph)=E$. Since $(-K_X+E)\cdot 
C >0$ for every curve $C\subset D\cap E$,
$\ph$ is finite 
of $D\cap E$. Thus if $F\subset E$ is a fiber of $\ph$, then $F\cap
D\neq \emptyset$ (because $D\cdot \NE(\ph)>0$), and $\dim (F\cap
D)=0$. This yields that $\dim F=1$, and by \cite[Theorem 2.3]{ando}
$\R_{\geq
  0}[f]$ is of 
type $(n-1,n-2)^{sm}$ and $Y$ is smooth. 

Finally  
 $-K_X+E=\ph^*(-K_{Y})$, thus $-K_{Y}$ is ample and
$Y$ is Fano (notice that $\NE(Y)$ is closed, 
because $\NE(Y)=\ph_*(\NE(X))$).
\end{proof}
\begin{lemma2}\label{blah}
Let $X$ be a Fano manifold and $D,E\subset X$ prime divisors with
$$\N(D\cap E,X)=\N(E,X)\cap
D^{\perp}\subseteq E^{\perp}.$$ 
Suppose that $E$ is a smooth $\pr^1$-bundle with fiber $f\subset E$,
such that $E\cdot f=-1$ and $D\cdot f>0$. 

Then $E\cong \pr^1\times F$ where $F$ is a Fano manifold, and $D\cap
E=\{pts\}\times F$.
Moreover the half-line $\R_{\geq 0}[f]$ is an extremal ray of type
$(n-1,n-2)^{sm}$, it is the unique extremal ray having negative
intersection with $E$, and the target of its contraction is Fano. 
\end{lemma2}
\begin{proof}
Consider the divisor  $D_{|E}$ in $E$. We have $\Supp (D_{|E})=D\cap
E$, and if $C\subseteq D\cap E$ is an irreducible curve,  
then $[C]\in\N(D\cap E,X)\subseteq D^{\perp}$, so that
$D_{|E}\cdot C=D\cdot C=0$.
Therefore $D_{|E}$ is nef.

Let $i\colon E\hookrightarrow X$ be the inclusion and take 
$\gamma\in\overline{\NE}(E)\cap
(D_{|E})^{\perp}$ with $\gamma\neq 0$. Then $i_*(\gamma)\in
\N(E,X)\cap D^{\perp}\subseteq E^{\perp}$, hence:
$$-K_{E}\cdot\gamma=-(K_X+E)\cdot i_*(\gamma)=-K_X\cdot
i_*(\gamma)=(-K_X)_{|E}\cdot\gamma>0.$$
By the contraction theorem,
 there exists a contraction $g\colon E\to Z$ 
such that $-K_{E}$ is $g$-ample and $\NE(g)=\overline{\NE}(E)\cap
(D_{|E})^{\perp}$ (see \cite[Theorem 3.7(3)]{kollarmori}).
Notice that  $D_{|E}\cdot f=D\cdot f>0$, hence $g$ does not
contract the fibers of 
the  $\pr^1$-bundle on $E$, and $\dim Z\geq 1$. 
 On the other hand
$g$
sends $D\cap E$ to a union of points, so that $\dim
Z=1$ by \ref{divisors}$(1)$. More precisely, since $g(f)=Z$, we get 
$Z\cong\pr^1$. The general fiber $F$ of $g$ is a Fano manifold of
dimension $n-2$, because  $-K_{E}$ is $g$-ample. 

By \cite[Lemma~4.9]{31} 
we conclude that
$E\cong \pr^1\times F$ and $g$ is the projection onto $\pr^1$. 
Since $D\cdot f>0$, $D\cap E$ dominates $F$ under the projection, and
is sent by $g$ to a union of points;
therefore $D\cap E=\{pts\}\times F$.
  
Using
Remark \ref{extremal}, we see that $\R_{\geq 0}[f]$ is an extremal ray of type
$(n-1,n-2)^{sm}$,  
and the target of its contraction is Fano. 

 Finally 
 let $R$ be an
 extremal ray of $X$ with $E\cdot R<0$. Then $R\subseteq
 \NE(E,X)\subseteq\NE(X)$,
 thus $R$ must be a one-dimensional face
 of $\NE(E,X)$.\footnote{Since $F$ and $E$ are Fano, 
  the cones $\NE(F)$, $\NE(E)$, $\NE(E,X)$, etc.\  
are closed and polyhedral.}  Since $E\cong\pr^1\times F$, we have
 $\NE(E)=\R_{\geq 0}[f]_E+\NE(\{pt\}\times F,E)$ and
 $\NE(E,X)= \R_{\geq 0}[f]+
\NE(\{pt\}\times F,X)$. On the other hand
 $\NE(\{pt\}\times F,X)\subset\N(\{pt\}\times F,X)=\N(D\cap
E,X)\subseteq E^{\perp}$,  
therefore $R= \R_{\geq 0}[f]$.
\end{proof}
\begin{remark2}\label{intersection}
Let $X$ be a projective manifold and $E_0\subset X$ a smooth prime
divisor which is a $\pr^1$-bundle with fiber $f_0\subset E_0$.  Let
$E_1,\dotsc,E_s\subset X$ be pairwise disjoint prime divisors such
that $E_0\neq E_i$ and
$E_0\cap E_i\neq\emptyset$ for every $i=1,\dotsc,s$. Then either
$E_1\cdot f_0=\cdots=E_s\cdot f_0=0$, or $E_i\cdot f_0> 0$ for $i=1,\dotsc,s$.
\end{remark2}
\begin{proof}
 For every $i=1,\dotsc,s$  we have $E_i\cdot f_0\geq 0$, because
 $E_0\neq E_i$.

Suppose that there exists $j\in\{1,\dotsc,s\}$ such that
$E_j\cdot f_0=0$. Since $E_0\cap E_j\neq\emptyset$, this
implies that $E_j$ contains a fiber $\overline{f}_0$ of the
$\pr^1$-bundle structure on $E_0$. If $i\in\{1,\dotsc,s\}$, $i\neq j$, we have
 $E_i\cap E_j=\emptyset$, in particular  $E_i\cap
\overline{f}_0=\emptyset$ and hence $E_i\cdot f_0=0$. 
\end{proof}
\begin{lemma2}\label{meno}
Let $X$ be a Fano manifold and $D\subset X$ a prime divisor with
$\codim\N(D,X)=c_X$. Let $E_1,\dotsc,E_s\subset X$ be
pairwise disjoint prime divisors such that:
$$D\cap E_i\neq\emptyset,\ \ D\neq E_i,\ \text{ and }\ \codim\N(D\cap
E_i,X)\leq c_X+1,\ \text{
for every }\ i=1,\dotsc,s.$$
If $s\geq 2$, then $\codim\N(D\cap
E_i,X)= c_X+1$ for every $i=1,\dotsc,s$, and
$$\N(D\cap E_i,X)=\N(D,X)\cap E_j^{\perp}\ \text{ for every }\ i\neq
j.$$
If $s\geq 3$, then there exists a linear subspace $L\subset\N(X)$, of
codimension $c_X+1$, such that
$L=\N(D\cap E_i,X)=\N(D,X)\cap E_i^{\perp}$
 for every $i=1,\dotsc,s$.
\end{lemma2}
\begin{proof}
Assume that $s\geq 2$, and let
 $i,j\in\{1,\dotsc,s\}$ with $i\neq j$. 
Since $E_i\cap E_j=\emptyset$, we have $\N(D\cap E_i,X)\subseteq
E_j^{\perp}$ by Remark \ref{elem}.  On the other hand,
since $D\cap E_j\neq\emptyset$ and $D\neq E_j$, 
there exists some curve $C\subset D$ with
$E_j\cdot C>0$, so that $\N(D,X)\not\subseteq E_j^{\perp}$. 
Therefore we get:
$$
\N(D\cap E_i,X)\subseteq \N(D,X)\cap E_j^{\perp}\subsetneq
\N(D,X),$$
hence $\rho_X-c_X-1\leq\dim\N(D\cap E_i,X)\leq\dim \N(D,X)\cap
E_j^{\perp}=\dim\N(D,X)-1= \rho_X-c_X-1$, and this yields the statement.

Assume now that $s\geq 3$, and set 
$L:=\N(D\cap E_1,X)$; the first part 
 already gives that
$\codim L=c_X+1$ and that $L=\N(D,X)\cap E_i^{\perp}$ for every $i=2,\dotsc,s$.
If $i,j\in\{2,\dotsc,s\}$ are distinct, 
again  by the first part we get
\begin{center}
$L=\N(D,X)\cap E_i^{\perp}=\N(D\cap E_j,X)=\N(D,X)\cap
E_1^{\perp}.$\end{center}
\end{proof}
\begin{lemma2}\label{zero}
Let $X$ be a Fano manifold and $D\subset X$ a prime divisor with
$\codim\N(D,X)=c_X$. 
Let $E_1,\dotsc,E_s\subset X$ be
pairwise disjoint smooth prime divisors, and suppose that
 $E_i$ is a $\pr^1$-bundle with  fiber  $f_i\subset E_i$, such that 
$E_i\cdot
f_i=-1$ and $D\cdot f_i>0$, for
every $i=1,\dotsc,s$.

Assume that $s\geq 2$.
Then 
 $\codim\N(E_i,X)=c_X$ and $\codim\N(D\cap E_i,X)=c_X+1$ for every
$i=1,\dotsc,s$; moreover
$\N(D\cap E_i,X)=\N(D,X)\cap E_j^{\perp}$ for every $i\neq
j$.
\end{lemma2}
\begin{proof}
Let $i\in\{1,\dotsc,s\}$.
We have $D\cap
E_i\neq\emptyset$ and $D\neq E_i$ because $D\cdot f_i>0$ and $E_i\cdot
f_i=-1$.
 Since $D\cdot
f_i>0$, by \ref{elem2}$(1)$ and by the definition of $c_X$
we have 
\renewcommand{\theequation}{\thethmdue}
\stepcounter{thmdue}
\begin{equation}\label{treno}
\codim\N(D\cap E_i,X)\leq
\codim\N(E_i,X)+1\leq c_X+1.
\end{equation} 
Therefore Lemma~\ref{meno} yields that  $\N(D\cap
E_i,X)=\N(D,X)\cap E_j^{\perp}$ if $i\neq j$, and $\codim\N(D\cap
E_i,X)=c_X+1$. By \eqref{treno} we get $\codim\N(E_i,X)=c_X$.
\end{proof}
\begin{lemma2}\label{ultimissimo}
Let $X$ be a Fano manifold and 
$D\subset X$ a prime divisor with
$\codim\N(D,X)=c_X$. Let
$E_1,\dotsc,E_s,\widehat{E}_1,\dotsc,\widehat{E}_s\subset X$ be prime
divisors such that $E_i$ and $\widehat{E}_i$ are
smooth $\pr^1$-bundles, with  fibers respectively  $f_i\subset E_i$
and $\widehat{f}_i\subset \widehat{E}_i$, and moreover: 
$$E_i\cdot
f_i=\widehat{E}_i\cdot
\widehat{f}_i=-1,\quad D\cdot f_i>0,\quad E_i\cdot
\widehat{f}_i>0,\quad \widehat{E}_i\cdot f_i>0,\quad 
[\widehat{f}_i]\not\in\N(E_i,X),$$
and no fiber $\widehat{f}_i$ is contained in $D$,
for
every $i=1,\dotsc,s$. We
assume also that
$E_1\cup\widehat{E}_1,\dotsc,E_s\cup\widehat{E}_s$ are pairwise
disjoint, and that $s\geq 2$. 

Then $\codim\N(E_i,X)=\codim\N(\widehat{E}_i,X)=c_X$ and
 $[f_i]\not\in\N(\widehat{E}_i,X)$ for every $i=1,\dotsc,s$.
\end{lemma2}
\begin{proof}
Lemma~\ref{zero} (applied to $D$ and $E_1,\dotsc,E_s$)
shows that $\codim\N(E_i,X)=c_X$ for every $i=1,\dotsc,s$.

Fix $i\in\{1,\dotsc,s\}$.
Since $\N(E_i\cap
\widehat{E}_i,X)\subseteq\N(E_i,X)$, we have
$[\widehat{f}_i]\not\in\N(E_i\cap 
\widehat{E}_i,X)$. Because
 ${E}_i\cdot \widehat{f}_i>0$, \ref{elem2}$(2)$
yields that
$\N(E_i\cap
\widehat{E}_i,X)$ has codimension $1$ in $\N(\widehat{E}_i,X)$. Recall
that by the definition of $c_X$ we have $\codim\N(\widehat{E}_i,X)\leq
c_X$, so that $\codim \N(E_i\cap 
\widehat{E}_i,X)\leq c_X+1$. 

Let us show that 
\renewcommand{\theequation}{\thethmdue}\stepcounter{thmdue}
\begin{equation}\label{nesima}
\codim\N(E_i\cap  \widehat{E}_i,X)=c_X+1\ \text{ 
and }\ \codim\N(\widehat{E}_i,X)=c_X.\end{equation} 
If
 $D\cap \widehat{E}_i=\emptyset$, then $\N(E_i\cap
\widehat{E}_i,X)\subseteq\N(E_i,X)\cap
D^{\perp}$ (see Remark \ref{elem}); on the other hand 
$\N(E_i,X)\cap
D^{\perp}\subsetneq\N(E_i,X)$, because $D\cdot f_i>0$. 
This yields  $\codim\N(E_i\cap
\widehat{E}_i,X)=c_X+1$. 

If instead
$D\cap \widehat{E}_i\neq\emptyset$, then $D\cdot\widehat{f}_i>0$,
because $D$ cannot contain any curve $\widehat{f}_i$.  
Thus we can apply Lemma~\ref{zero}
to the divisors $D$ and
$E_1,\dotsc,E_{i-1},\widehat{E}_i,E_{i+1},\dotsc,E_{s}$, and we deduce that
$\codim\N(\widehat{E}_i,X)=c_X$. Hence we have \eqref{nesima}.
 
Since $\widehat{E}_i\cdot f_i>0$ and
$\codim\N(E_i,X)=c_X=\codim\N(E_i\cap  \widehat{E}_i,X)-1$, 
 again by \ref{elem2}$(2)$ we get
 $[f_i]\not\in \N(E_i\cap
\widehat{E}_i,X)$. For dimensional reasons $\N(E_i\cap
\widehat{E}_i,X)=\N(E_i,X)\cap\N( \widehat{E}_i,X)$, and we conclude
that $[f_i]\not\in\N( \widehat{E}_i,X)$.
\end{proof}
\subsection{The case where $X$ is a product}\label{geq4}
 The main results of this section are the following. 
\begin{proposition2}\label{primameta}
Let $X$ be a Fano manifold  such that either $c_X\geq 4$, or $c_X=3$ and 
$X$ satisfies (\ref{caso}.a). 

Then  $X\cong S\times T$, where
$S$ is a Del Pezzo surface with  $\rho_S=c_X+1$, and $c_T\leq c_X$. In
particular, $c_X\leq 8$.
\end{proposition2}
\begin{corollary2}\label{asilo}
Let $X$ be a Fano manifold with $c_X=3$. Then $X$ satisfies (\ref{caso}.b).
\end{corollary2}
\begin{proof}[Proof of Corollary \ref{asilo}]
 By contradiction, suppose that $X$ satisfies
(\ref{caso}.a). Then by Proposition \ref{primameta}
we have $X\cong S\times T$ and $\rho_S=4$, \emph{i.e.}\ $S$ is the
blow-up of $\pr^2$ in three non-collinear points. Consider the sequence:
$$X\cong S\times T\longrightarrow S_1\times T\longrightarrow \mathbb{F}_1\times T
\longrightarrow 
\pr^1\times T,$$
where 
$S_1$ is the blow-up of $\pr^2$ in two distinct points. 
Let $C\subset\mathbb{F}_1$ be the section of the $\pr^1$-bundle
containing the two points blown-up under $S\to\mathbb{F}_1$. Let
moreover $\widetilde{C}\subset S$ be its transform, 
and $D:=\widetilde{C}\times
T\subset X$. Then $\codim\N(D,X)=3$, and the sequence above is a special
Mori program for $-D$. 
The image of $D$ in $\mathbb{F}_1\times T$ is 
$C\times T$, and $\N(C\times T,\mathbb{F}_1\times T)\subsetneq
\N(\mathbb{F}_1\times T)$.
Thus we have a contradiction with (\ref{caso}.a).
\end{proof}
\begin{parg2}{\bf Outline of the proof of  Proposition
    \ref{primameta}.} \label{secondooutline} 
There are three preparatory steps, and then the actual proof.

The first step is to apply the construction of section~\ref{run} to a
prime divisor $D\subset X$ with $\codim\N(D,X)=c_X$. We consider a
special Mori program for $-D$, and this determines pairwise
disjoint $\pr^1$-bundles $E_1,\dotsc,E_s\subset X$ as in
Lemma~\ref{pluto}; we denote by $f_i\subset E_i$ a fiber. 
The crucial property here is that $s\geq 3$: indeed
 $s\geq \codim\N(D,X)-1=c_X-1$, so that $s\geq 3$ if $c_X\geq
4$. On the other hand if 
$c_X=3$ we have $s=3$ by (\ref{caso}.a).
  Then for $i=1,\dotsc,s$ we show that $\codim\N(E_i,X)=c_X$ and that
 $\R_{\geq 0}[f_i]$ is an extremal ray of type $(n-1,n-2)^{sm}$, such that
the target of its contraction is again Fano. This is Lemma~\ref{a}. 

\medskip

In particular, this shows that $X$ has at least one extremal ray
 $R_0$ of type $(n-1,n-2)^{sm}$ such that if
$E_0:=\Lo(R_0)$, then $\codim\N(E_0,X)=c_X$, and the target of the
contraction of $R_0$ is Fano. 

Now
 we replace $D$ by $E_0$, and
apply again the same construction.
 Let $p\colon E_0\to F$ be the
$\pr^1$-bundle structure. 
Since $E_1,\dotsc,E_s$ are pairwise disjoint, either $E_0\cap E_i$ is
a union of fibers of $p$ for every $i=1,\dotsc,s$, or  $p(E_0\cap
E_i)=F$  
for every $i=1,\dotsc,s$. 
The second preparatory step is to show that if  $E_1,\dotsc,E_s$
  intersect $E_0$ horizontally with respect to the
 $\pr^1$-bundle (\emph{i.e.}\  $p(E_0\cap
E_i)=F$), 
 the divisors
$E_0,\dotsc,E_s$ have very special properties; in particular,  for
every $i=0,\dotsc,s$,
$E_i\cong
\pr^1\times F$ where $F$ is an $(n-2)$-dimensional Fano manifold. This
is  Lemma~\ref{tre}.  

\medskip

The third preparatory step is
 show  that we can always choose the extremal ray $R_0$, and the
 special Mori program for $-E_0$, in such a way that $E_1,\dotsc,E_s$
 actually intersect $E_0$ horizontally with respect to the
 $\pr^1$-bundle, so that the previous result applies. This is
 Lemma~\ref{due}. 

\medskip

Then we are ready for the proof of Proposition \ref{primameta}.
We use the  the properties given by Lemma~\ref{tre} to show that
$E_1,\dotsc,E_s$ are the exceptional divisors of the 
blow-up $\sigma\colon X\to X_s$ of a Fano manifold $X_s$ in $s$ smooth
codimension $2$ subvarieties. Moreover
there is an elementary contraction of
fiber type $\ph\colon X_s\to Y$ such that if
$\psi:=\ph\circ\sigma\colon X\to Y$,
then $\psi(E_0)=Y$, and $\psi$ is finite on $\{pt\}\times F\subset
E_0$ (recall that $E_0\cong\pr^1\times F$). 
We have then two possibilities: either
$\psi$ is not finite on $E_0$ and $\dim Y=n-2$, or $\psi$ is finite on
$E_0$ and $\dim Y=n-1$.

We first consider the case where $\psi$  is not finite on
$E_0$, in \ref{316}. 
We use the divisors $E_0,\dotsc,E_s$ to define a contraction
$X\to S$ onto a surface, such that the induced morphism $\pi\colon 
X\to S\times
Y$ is finite. Finally we  show that in fact $\pi$ is an isomorphism;
here the key property is that $E_0,\dotsc,E_s$ are
products.

Then we consider in \ref{second}
the case where $\psi$ is finite on $E_0$. In this
situation $Y$ is smooth, and both $\psi$ and $\ph$ are conic
bundles. If $T_1,\dotsc,T_s\subset X_s$ are the subvarieties blown-up
by $\sigma$, the transforms $\widehat{E}_1,\dotsc,\widehat{E}_s\subset
X$ of $\ph^{-1}(\ph(T_i))$ are smooth $\pr^1$-bundles. 

Similarly to what previously done for $E_0,\dotsc,E_s$, we show that
$\widehat{E}_i\cong\pr^1\times F$ for 
every $i=1,\dotsc,s$. 

Since $\psi(E_0)=Y$, $Y$ is covered by the family of rational curves 
$\psi(\pr^1\times\{pt\})$. We use a result from \cite{unsplit} to show
that in fact these 
rational curves are the fibers of a smooth morphism $Y\to
Y'$, where $\dim Y'=n-2$.

In this way we get a contraction $X\to Y'$, and we proceed similarly
to the previous case: we use the divisors
$E_0,E_1,\dotsc,E_s,\widehat{E}_1,\dotsc,\widehat{E}_s$ to define a
contraction $X\to S$ onto a surface, and show that the induced
morphism $X\to S\times Y'$ is an isomorphism. 
\end{parg2}
Let us start with the first preparatory result.
\begin{lemma2}\label{a}
Let $X$ be a Fano manifold  such that either $c_X\geq 4$, or $c_X=3$ and 
$X$ satisfies (\ref{caso}.a). 

Let $D\subset X$ be a prime divisor 
with $\codim\N(D,X)=c_X$, consider a special Mori program for
$-D$, and
let $E_1,\dotsc, E_s \subset X$ be the $\pr^1$-bundles determined by
 the Mori program. For $i=1,\dotsc,s$ let
$f_i\subset E_i$ be a fiber of the $\pr^1$-bundle, and
 set $R_i:=\R_{\geq 0}[f_i]$. Then we have the following:
\begin{enumerate}[$(1)$]
\item $s\in\{c_X-1,c_X\}$ and $s\geq 3$;
\item $R_i$ is an extremal ray of type
$(n-1,n-2)^{sm}$,
 the target of the contraction of $R_i$ is Fano, and
$\codim\N(E_i,X)=c_X$, for every
  $i=1,\dotsc,s$;
\item there exists a linear subspace $L\subset\N(X)$, of codimension
  $c_X+1$, such that
$$L=\N(D\cap E_i,X)=\N(D,X)\cap E_i^{\perp}=\N(E_i,X)\cap
  E_i^{\perp}\ \text{ for every }i=1,\dotsc,s.$$ 
\end{enumerate}
\end{lemma2}
\noindent We will call $R_1,\dotsc,R_s$ {\bf the extremal rays
  determined by the special Mori program for $-D$} that we are
  considering. Notice that differently from the case of the
$\pr^1$-bundles $E_1,\dotsc,E_s$, 
the extremal rays $R_1,\dotsc,R_s$ are defined only
when $X$ satisfies the assumptions of Lemma~\ref{a}, and $D\subset X$
is a prime divisor with $\codim\N(D,X)=c_X$.
\begin{proof}
We know by Lemma~\ref{pluto} that:
 $E_i\cdot
f_i=-1$ and $D\cdot f_i>0$ for $i=1,\dotsc,s$, 
$E_1,\dotsc,E_s$ are
pairwise disjoint, and
 $s\in\{c_X-1,c_X\}$
 because $\codim\N(D,X)=c_X$. Moreover, if $c_X=3$, then $s=3$ by
 (\ref{caso}.a), so that in any case $s\geq 3$, and we get $(1)$.

Therefore, by Lemma~\ref{zero}, we have $\codim\N(E_i,X)=c_X$ and
$\codim\N(D\cap E_i,X)=c_X+1$ for every
$i=1,\dotsc,s$. In particular, 
 Lemma~\ref{meno} applies; let $L\subset\N(X)$ be the
linear subspace such that $\codim L=c_X+1$ and
$L=\N(D\cap E_i,X)=\N(D,X)\cap
E_i^{\perp}$ for every 
 $i=1,\dotsc,s$. 

Fix $i\in\{1,\dotsc,s\}$.
Since  $E_i\cdot f_i=-1$, we have $\N(E_i,X)\not\subseteq
E_i^{\perp}$, therefore
$\dim\N(E_i,X)\cap E_i^{\perp}=\dim\N(E_i,X)-1=\rho_X-c_X-1=\dim
L$. On the other hand  
we have $L\subseteq E_i^{\perp}$ and 
$L=\N(D\cap E_i,X)$, in particular $L\subseteq\N(E_i,X)$. Thus
$L\subseteq\N(E_i,X)\cap E_i^{\perp}$, so the two subspaces must
coincide, and
we get $(3)$.

Finally, $(2)$ follows from
 Remark \ref{extremal} applied to $D$ and $E_i$.
\end{proof}
\begin{lemma2}\label{expl}
Let $X$ be a Fano manifold  such that either $c_X\geq 4$, or $c_X=3$ and 
$X$ satisfies (\ref{caso}.a). 

 Let $D\subset X$ be a prime divisor 
with $\codim\N(D,X)=c_X$, and $R$ an extremal ray of type
$(n-1,n-2)^{sm}$ such that $D\cdot R>0$, $R\not\subset\N(D,X)$, and
the target of the contraction of $R$ is Fano.  

Set $E:=\Lo(R)$. Then $\N(D\cap E,X)=\N(D,X)\cap E^{\perp}=\N(E,X)\cap E^{\perp}$.
\end{lemma2}
\begin{proof}
Consider the contraction $\ph\colon X\to Y$ of $R$, so that by the
assumptions $Y$ is a Fano manifold, and 
consider the prime divisor
$\ph(D)\subset Y$.

By Proposition \ref{tomm}, there exists 
 a special Mori program for $-\ph(D)$ in $Y$. Together with $\ph$,
 this gives a special Mori program for $-D$ in $X$, where the first
 extremal ray is precisely $Q_0=R$:
$$
X\stackrel{\ph}{\longrightarrow}Y=Y_0
\stackrel{\sigma_{0}}{\dasharrow} 
Y_1 \dasharrow\quad\cdots
\quad
\dasharrow Y_{k-1}\stackrel{\sigma_{k-1}}{\dasharrow} Y_k.
$$
We apply Lemmas \ref{pluto} and \ref{a}; 
since $R\not\subset\N(D,X)$, $E$ is one of the $\pr^1$-bundles
determined by this special Mori program for $-D$.
 Thus the statement
follows from \ref{a}(3).
\end{proof}
\begin{remark2}\label{inizio}
Let $X$ be a Fano manifold  such that either $c_X\geq 4$, or $c_X=3$ and 
$X$ satisfies (\ref{caso}.a). 
 Recall from Proposition \ref{tomm} that
there exists a special Mori program for any divisor in $X$.

The first consequence of
Lemma~\ref{a} (applied to any prime divisor $D\subset X$ with
$\codim\N(D,X)=c_X$) is that $X$ has an
extremal ray $R_0$ of type $(n-1,n-2)^{sm}$ such that if
$E_0:=\Lo(R_0)$, then $\codim\N(E_0,X)=c_X$, and the target of the
contraction of $R_0$ is Fano. 

In particular, we can consider a special Mori program for $-E_0$, and
apply again Lemma \ref{a}.
Let $R_1,\dotsc,R_s$ be the extremal rays determined by the
Mori program,
with loci $E_1,\dotsc,E_s$.
Since, by \ref{pluto}$(3)$ and \ref{pluto}$(4)$,
 $E_1,\dotsc,E_s$ are pairwise disjoint and $E_0\neq E_i$, $E_0\cap
E_i\neq\emptyset$ for $i=1,\dotsc,s$, by
Remark \ref{intersection}
we have two
possibilities: either
 $E_1\cdot R_0=\cdots=E_s\cdot
R_0=0$, or
 $E_i\cdot
R_0>0$ for every $i=1,\dotsc,s$.
\end{remark2}
In the next Lemma we are going to show that in the second case
(\emph{i.e.}\ when $E_1\cdot R_0>0$) the extremal rays $R_0,\dotsc,R_s$ have  
very special properties, in particular that the divisors
$E_0,\dotsc,E_s$ are products.
\begin{lemma2}\label{tre}
Let $X$ be a Fano manifold  such that either $c_X\geq 4$, or $c_X=3$ and 
$X$ satisfies (\ref{caso}.a). 

Let $R_0$ be an extremal ray of $X$, 
of type $(n-1,n-2)^{sm}$, such that
the target of the contraction of $R_0$ is Fano, and
$\codim\N(E_0,X)=c_X$, where
 $E_0:=\Lo(R_0)$.

Consider a special Mori program for $-E_0$, let
$R_1,\dotsc,R_{s}$ be the extremal rays determined by the Mori program, and
set $E_i:=\Lo(R_i)$ for $i=1,\dotsc,s$.

Assume that $E_1\cdot R_0>0$. Then we have the following:
\begin{enumerate}[$(1)$]
\item  $\codim\N(E_i,X)=c_X$, 
and $E_i\cong \pr^1\times F$ with $F$ 
an $(n-2)$-dimensional Fano
manifold, for $i=0,\dotsc,s$.
We set $F_i:=\{pt\}\times F\subset E_i$;
\item $R_i$ is the unique
extremal ray of $X$ having negative intersection with $E_i$, and
the target of the contraction of $R_i$ is Fano, for every $i=0,\dotsc,s$;
\item $E_1,\dotsc,E_s$ are pairwise disjoint, and $E_0\cap
  E_i=\{pts\}\times F$ for every $i=1,\dotsc,s$;
\item $E_i\cdot R_0>0$ and $E_0\cdot R_i>0$ 
for every $i=1,\dotsc,s$; 
\item 
there exists a linear subspace $L\subset\N(X)$, of codimension
  $c_X+1$, such that
$$L=\N(E_0\cap E_i,X)=\N(F_j,X)\ \text{ 
and }\ \N(E_j,X)= \R R_j\oplus L$$
for every $i=1,\dotsc,s$ and
  $j=0,\dotsc,s$, and moreover $\dim(\R(R_0+\cdots+R_s)+ L)=s+1+\dim L$;
\item $L\subseteq E_0^{\perp}\cap\cdots\cap E_s^{\perp}$, and equality
  holds 
if $s=c_X$.
\end{enumerate} 
\end{lemma2}
\begin{proof}
By \ref{pluto}(3) and \ref{pluto}(4) we know that
 $E_0\cdot R_i>0$ (in particular $E_0\neq E_i$ and $E_0\cap
E_i\neq\emptyset$) and $R_i\not\subset\N(E_0,X)$
for $i=1,\dotsc,s$, and that
$E_1,\dotsc,E_s$ are pairwise disjoint.

Secondly, Lemma~\ref{a} shows that 
 $s\in\{c_X-1,c_X\}$ and $s\geq 3$, that
$\codim\N(E_i,X)=c_X$ for $i=1,\dotsc,s$, and that
 there exists a linear subspace $L\subset\N(X)$, of codimension
  $c_X+1$, such that
\renewcommand{\theequation}{\thethmdue}\stepcounter{thmdue}
\begin{equation}\label{utile}
L=\N(E_0\cap E_i,X)=\N(E_0,X)\cap E_i^{\perp}=\N(E_i,X)\cap
  E_i^{\perp}\end{equation}
for every $i=1,\dotsc,s.$ Moreover Remark \ref{intersection}
yields $E_i\cdot R_0>0$ for every $i=1,\dotsc,s$, because $E_1\cdot
R_0>0$, so we get $(4)$.

 Fix $i\in\{1,\dotsc,s\}$.
We have $\dim \N(E_0\cap E_i,X)=\dim L=\rho_X-c_X-1<\rho_X-c_X=
\dim \N(E_0,X)$, and
since $E_i\cdot R_0>0$,
\ref{elem2}$(2)$ gives
$R_0\not\subset\N(E_0\cap E_i,X)$. Moreover
$$\N(E_0\cap
E_i,X)\subseteq\N(E_0,X)\cap\N(E_i,X)\subsetneq\N(E_0,X)$$ 
(because $R_i\not\subset\N(E_0,X)$),
and since 
$\N(E_0\cap
E_i,X)$ has codimension $1$ in $\N(E_0,X)$, we deduce that
$\N(E_0\cap
E_i,X)=\N(E_0,X)\cap\N(E_i,X)$. This yields that
$R_0\not\subset\N(E_i,X)$.

Now we can apply Lemma~\ref{expl} to $E_i$ and $R_0$, and deduce that 
\renewcommand{\theequation}{\thethmdue}\stepcounter{thmdue}
\begin{equation}\label{utile2}
L=\N(E_0\cap E_i,X)=\N(E_i,X)\cap E_0^{\perp}.\end{equation} 

Thanks to (4), \eqref{utile}, and \eqref{utile2},
we can use Lemma~\ref{blah} to show $(1)$.
First of all we apply 
Lemma~\ref{blah} with $D=E_i$ and $E=E_0$, and we deduce that
$E_0\cong\pr^1\times F$ where $F$ is an $(n-2)$-dimensional Fano
manifold, and $E_0\cap E_i=\{pts\}\times F\subset
E_0$. Moreover we get $(2)$ for $R_0$.

 Then we apply Lemma~\ref{blah} again, with
$D=E_0$ and $E=E_i$, and we get $E_i\cong\pr^1\times F^i$ and $E_0\cap
E_i=\{pts\}\times F^i\subset E_i$; in particular, $F^i=F$, and we have
$(3)$. Moreover we 
get $(2)$ for $R_i$.

We have $L\subseteq E_0^{\perp}\cap\cdots\cap E_s^{\perp}$ by
\eqref{utile} and \eqref{utile2}.
To get $(5)$, it is enough to show that
$[f_0],\dotsc,[f_s]\in\N(X)$ are linearly independent and that
$\R([f_0]+\cdots+[f_s])\cap L=\{0\}$.  
So suppose that 
there exist $\lambda_0,\dotsc,\lambda_s\in\R$ such that
$$\sum_{i=0}^s\lambda_i f_i\in L.$$
Intersecting with $E_j$ for $j\in\{1,\dotsc,s\}$ we get
$\lambda_j=\lambda_0 E_j\cdot f_0$, and intersecting with $E_0$ we get
$\lambda_0(\sum_{i=1}^s(E_i\cdot f_0)(E_0\cdot f_i)-1)= 0$. Since
$E_i\cdot f_0$ and $E_0\cdot f_i$ are positive integers by (4), and
$s\geq 3$,
 we get $\lambda_0=0$  and
 hence $\lambda_i=0$ for $i=1,\dotsc,s$, and we are done.

We are left to show $(6)$. Similarly to what we have done for
$[f_0],\dotsc,[f_s]$, one checks that 
 $[E_0],\dotsc,[E_{s}]$ are
linearly independent in $\mathcal{N}^1(X)$, so that $\codim (E_0^{\perp}\cap\cdots
\cap E_{s}^{\perp})=s+1$. Since
$L\subseteq E_0^{\perp}\cap\cdots
\cap E_{s}^{\perp}$ and $\codim L=c_X+1$, if $s=c_X$ the two subspaces coincide.
\end{proof}
\begin{lemma2}\label{due}
Let $X$ be a Fano manifold  such that either $c_X\geq 4$, or $c_X=3$ and 
$X$ satisfies (\ref{caso}.a). 
Then $X$ has an extremal ray $R_0$ with the following properties:
\begin{enumerate}[$\bullet$]
\item $R_0$ is of type $(n-1,n-2)^{sm}$,
the target of the contraction of $R_0$ is Fano, and
$\codim\N(E_0,X)=c_X$, where
 $E_0:=\Lo(R_0)$;
\item there exists a special Mori program for $-E_0$ such that,
if $R_1,\dotsc,R_{s}$ are the  extremal rays
determined by the Mori program,
 we have
have $\Lo(R_i)\cdot R_0>0$ for every $i=1,\dotsc,{s}$.
\end{enumerate}
\end{lemma2}
\begin{proof}
Let $\mathcal{S}=\{S^1,\dotsc,S^h\}$ be 
an ordered set of 
extremal rays of $X$, and set $E^i:=\Lo(S^i)$.
Consider the following properties:
\begin{enumerate}[(P1)]
\item $S^i$ is  of type $(n-1,n-2)^{sm}$, the target of the
  contraction of $S^i$ is Fano, and 
$\codim \N(E^i,X)=c_X$, for every $i=1,\dotsc,h$;
\item
$E^{i-1}\cdot S^i>0$ and $S^{i}\not\subset\N(E^{i-1},X)$, for every
$i=2,\dotsc,h$;
\item for every $1\leq
j<i\leq h$ we have
$E^i\cdot S^j=0$ and $E^i\cap E^j\neq\emptyset$.
\end{enumerate}

We notice first of all that by Remark \ref{inizio}, there exists an
extremal ray $S^1$ of $X$, of type $(n-1,n-2)^{sm}$, such that
$\codim\Lo(S^1)=c_X$, and the target of the contraction of $S^1$ is
Fano. Then $\mathcal{S}=\{S^1\}$ satisfies properties (P1), (P2), and (P3).

Consider now an arbitrary ordered set of extremal rays
$\mathcal{S}=\{S^1,\dotsc,S^h\}$ satisfying properties (P1), (P2), and (P3).
We show that $h\leq\rho_X$.

Let $\gamma_i\in S^i$ a non-zero element, for $i=1,\dotsc,h$. 
We have $E^i\cdot\gamma_i\neq 0$ for every $i=1,\dotsc,h$, and
$E^i\cdot\gamma_j=0$ for every $1\leq j<i\leq h$ by (P3). This shows
that 
 $\gamma_1,\dotsc,\gamma_h$ are linearly independent in $\N(X)$:
indeed if there exists $a_1,\dotsc,a_h\in\R$ such that
$\sum_{i=1}^ha_i\gamma_i=0$, then intersecting with $E_h$ we get
$a_h=0$, and so on. Thus $h\leq\rho_X$.

Then Lemma~\ref{due} is a consequence of the following claim.
\end{proof}
\begin{claim2}\label{demi}
Assume that  $\mathcal{S}=\{S^1,\dotsc,S^h\}$ is an ordered set of extremal rays
having
  properties (P1), (P2), and (P3). Then either  $R_0:=S^h$
satisfies the statement of Lemma~\ref{due}, 
or there exists an extremal ray $S^{h+1}$
such that $\mathcal{S}':=\{S^1,\dotsc,S^h,S^{h+1}\}$ still has
properties (P1), (P2), and (P3). 
\end{claim2}
\begin{proof}[Proof of Claim~\ref{demi}]
By (P1) the ray $S^h$ is of type $(n-1,n-2)^{sm}$, the target of its
contraction is Fano, and  
 $\codim\N(E^h,X)=c_X$.  Consider a
special Mori program for $-E^h$ (which exists by  Proposition \ref{tomm}), 
and let $S^{h+1}_1,\dotsc,S^{h+1}_s$ be the extremal rays determined
by the Mori program, as in Lemma~\ref{a}.
Notice that $s\geq 3$ by
\ref{a}(1).
We set $E^{h+1}_l:=\Lo(S^{h+1}_l)$ for
$l=1,\dotsc,s$, so that $E^{h+1}_1,\dotsc,E^{h+1}_s$ are the
$\pr^1$-bundles determined by the Mori program. By
\ref{pluto}$(3)$ we have
\renewcommand{\theequation}{\thethmdue}\stepcounter{thmdue}
\begin{equation}\label{recall}
E^h\cdot S^{h+1}_l>0\ \text{ and
}\ S^{h+1}_l\not\subset\N(E^h,X)\ \text { for every }l=1,\dotsc,s, 
\end{equation}
and $E^{h+1}_1,\dotsc,E^{h+1}_s$ are pairwise disjoint by \ref{pluto}$(4)$.

Remark \ref{inizio} shows that the intersections $E^{h+1}_l\cdot S^h$ (for
$l=1,\dotsc,s$) are either 
all zero, or all positive. In the latter case, $S^h$ satisfies the
statement of Lemma~\ref{due}.

Thus let us assume that  $E^{h+1}_1\cdot S^h=\cdots=E^{h+1}_s\cdot
S^h=0$, and set  $S^{h+1}:=S^{h+1}_1$
and $E^{h+1}:=E^{h+1}_1$.  

Since by assumption $\mathcal{S}$ has properties (P1) and (P2), in
order to show that $\mathcal{S}'$ still satisfies (P1) and (P2), we
just have to consider the case $i=h+1$. Then (P2) is given by
\eqref{recall}, and (P1) follows from \ref{a}(2). 

Now let us show 
the following:
\renewcommand{\theequation}{\thethmdue}\stepcounter{thmdue}
\begin{equation}\label{orata}
E_l^{h+1}\cdot S^j=0\ \text{ and }
\ E_l^{h+1}\cap E^j\neq\emptyset\quad\text{for every }j=1,\dotsc,h\text{
and }l=1,\dotsc,s.
\end{equation}
In particular, for $l=1$, \eqref{orata} implies
 that  $\mathcal{S}'$ satisfies (P3).

Let $l\in\{1,\dotsc,s\}$. Since $E^h\cdot S_l^{h+1}>0$ by \eqref{recall},
we have $E^h\cap E_l^{h+1}\neq 
\emptyset$; moreover we have assumed that $E_l^{h+1}\cdot S^h=0$.
Therefore \eqref{orata} holds for $j=h$ and $l=1,\dotsc,s$. 

We proceed by decreasing induction on $j$: we assume that
\eqref{orata} holds 
for some $j\in\{2,\dotsc,h\}$ and for every $l=1,\dotsc,s$,
and we show that 
$E_l^{h+1}\cdot S^{j-1}=0$ and $E_l^{h+1}\cap E^{j-1}\neq\emptyset$
for every  $l=1,\dotsc,s$.

\medskip

Fix $l\in\{1,\dotsc,s\}$. Since $E_l^{h+1}\cdot S^j=0$ and
$E_l^{h+1}\cap E^j\neq\emptyset$ by the induction assumption,
$E_l^{h+1}$ contains a
curve $C$ with class in $S^{j}$,
 in particular
\renewcommand{\theequation}{\thethmdue}\stepcounter{thmdue}
\begin{equation}\label{preciso}
S^j\subset\N(E_l^{h+1},X).
\end{equation}
Since 
 $E^{j-1}\cdot S^{j}>0$ by (P2), we have $E^{j-1}\cap C\neq\emptyset$ and
 hence $E_l^{h+1}\cap
E^{j-1}\neq\emptyset$. 
Moreover $E_l^{h+1}\cdot S^{j}=0$ implies that $E_l^{h+1}\neq E^{j-1}$, thus
$E_l^{h+1}\cdot S^{j-1}\geq 0$. 

Recall from (P1) that $E^{j-1}$ is the locus of the extremal ray $S^{j-1}$, of
type $(n-1,n-2)^{sm}$; in particular $E^{j-1}$ is a $\pr^1$-bundle.
Since $E^{h+1}_1,\dotsc,E^{h+1}_{s}$ are pairwise disjoint,
by Remark \ref{intersection}  the intersections 
$E_l^{h+1}\cdot S^{j-1}$ (for $l=1,\dotsc,s$)
are either all zero or all positive. 

By contradiction, suppose that $E_l^{h+1}\cdot S^{j-1}>0$ for every
$l=1\dotsc,s$. We have $\codim\N(E^{j-1},X)=c_X$ by (P1), hence
 \ref{elem2}$(1)$ gives
$$\codim\N(E^{j-1}\cap E^{h+1}_l,X)\leq \codim\N(E^{j-1},X)+1=c_X+1\ 
\text{ for every }l=1,\dotsc,s.$$
Since $s\geq 3$, we can apply Lemma~\ref{meno} to $E^{j-1}$ and
$E^{h+1}_1,\dotsc,E^{h+1}_s$, and deduce that $\codim\N(E^{j-1}\cap
E^{h+1},X)=c_X+1$ and $\N(E^{j-1}\cap
E^{h+1},X)\subseteq (E^{h+1})^{\perp}$.
In particular
  $$\N(E^{j-1}\cap
E^{h+1},X)\subseteq\N(E^{h+1},X)\cap (E^{h+1})^{\perp}.$$
On the other hand $\N(E^{h+1},X)\not\subseteq (E^{h+1})^{\perp}$
because $E^{h+1}\cdot S^{h+1}<0$, therefore
$$\codim\left(\N(E^{h+1},X)\cap (E^{h+1})^{\perp}\right)=c_X+1=\codim 
\N(E^{j-1}\cap
E^{h+1},X),$$  and the two subspaces coincide. 

By \eqref{preciso} and by the induction assumption we have
 $S^{j}\subset \N(E^{h+1},X)\cap
(E^{h+1})^{\perp}$, therefore $S^j\subset\N(E^{j-1},X)$, and this
contradicts property (P2).
\end{proof}
\begin{proof}[Proof of Proposition \ref{primameta}]
  Let  $R_0$ be the extremal ray of $X$
given by Lemma~\ref{due}, and set
$E_0:=\Lo(R_0)$. Then $\codim\N(E_0,X)=c_X$, and 
 there exists a special Mori program for $-E_0$ which determines extremal rays
 $R_1,\dotsc,R_{s}$ such that
$E_i\cdot R_0>0$ for all $i=1,\dotsc,{s}$, where $E_i:=\Lo(R_i)$.
Thus Lemma~\ref{tre}
applies.

\medskip

If $R$ is an extremal ray of $X$ different from $R_1,\dotsc,R_s$,
by \ref{tre}$(2)$ we have $E_i\cdot R\geq 0$ for every $i=1,\dotsc,s$, hence 
$(-K_X+E_1+\cdots+E_{s})\cdot R>0$. On the other hand
$(-K_X+E_1+\cdots+E_{s})\cdot R_i=0$ for every $i=1,\dotsc,s$
(recall from \ref{tre}(3) that $E_1,\dotsc,E_s$ are pairwise
disjoint), 
therefore $-K_X+E_1+\cdots+E_{s}$ is nef and 
$$(-K_X+E_1+\cdots+E_{s})^{\perp}\cap\NE(X)=R_1+\cdots+R_{s}$$ 
is a face of $\NE(X)$, of dimension $s$  
by \ref{tre}$(5)$.

Let   $\sigma\colon X\to X_{s}$ be the associated
contraction, so that $\ker\sigma_*=\R(R_1+\cdots+R_s)$.
 Since $E_1,\dotsc,E_s$ are pairwise disjoint, 
we see that $\Exc(\sigma)=E_1\cup\cdots\cup E_s$, $X_s$ is smooth, and 
$\sigma$ is the
blow-up of $s$ smooth, pairwise disjoint, irreducible
 subvarieties $T_1,\dotsc,T_{s}\subset X_{s}$ of codimension $2$, 
where $T_i:=\sigma(E_i)$
 for $i=1,\dotsc,s$. 
Moreover $X_{s}$
is again Fano, because
$-K_X+E_1+\cdots+E_{s}=\sigma^*(-K_{X_s})$. 
Recall from \ref{tre}$(1)$ that
 $E_i\cong\pr^1\times F$, and notice that $\sigma_{|E_i}$ is the
 projection onto $F\cong T_i$.

Set $(E_0)_s:=\sigma(E_0)\subset X_s$.
Since $E_0\cong\pr^1\times F$
and $E_0\cap E_i=\{pts\}\times F$ for $i=1,\dotsc,s$
by \ref{tre}$(1)$ and \ref{tre}$(3)$,
the morphism $\sigma_{|E_0}\colon E_0\to (E_0)_s$ is birational and
finite, \emph{i.e.}\ it is the normalization. Moreover for
$i=1,\dotsc,s$
we have $T_i=\sigma(E_0\cap E_i)\subset (E_0)_s$, so that
\renewcommand{\theequation}{\thethmdue}\stepcounter{thmdue}
\begin{equation}\label{sab0}
\N(T_i,X_s)=\sigma_*\left(\N(E_0\cap E_i,X)\right)=\sigma_*(L),
\end{equation}
where $L\subset\N(X)$ is the linear subspace defined in \ref{tre}(5).
Again by 
\ref{tre}$(5)$ we know that $\N(E_0,X)=\R R_0\oplus L$, and that
$\dim(\ker\sigma_*+ \N(E_0,X))=\dim\ker\sigma_*+\dim \N(E_0,X)$, therefore:
\renewcommand{\theequation}{\thethmdue}\stepcounter{thmdue}
\begin{equation}\label{sab1}
\ker\sigma_*\cap \N(E_0,X)=\{0\}\
\text{ and }
\ 
\N((E_0)_s,X_s)=\R\sigma_*(R_0)\oplus \sigma_*(L).
\end{equation}

Finally,
since $\sigma^*((E_0)_s)=E_0+\sum_{i=1}^s(E_0\cdot f_i)E_i$ (as usual
we denote by $f_i\subseteq E_i$ a 
fiber of the $\pr^1$-bundle),
 by \ref{tre}$(4)$ and  \ref{tre}$(6)$   we see that 
\renewcommand{\theequation}{\thethmdue}\stepcounter{thmdue}
\begin{equation}\label{sab3}
(E_0)_s\cdot\sigma(f_0)=\sum_{i=1}^s(E_0\cdot f_i)(E_i\cdot f_0)-1>0\ 
\text{ and }\ \sigma_*(L)\subseteq
(E_0)_s^{\perp}
\end{equation}
 (recall that $s\geq 3$ and $s\in\{c_X-1,c_X\}$ by \ref{a}(1)).

\medskip

Factoring $\sigma$ as a sequence of $s$ blow-ups, we can view 
$\sigma\colon X\to X_s$
as a part of a special Mori program for $-E_0$ in $X$, with $s$ steps,
and by \eqref{sab1}
at each step we have $Q_i\not\subset\N((E_0)_i,X_i)$. In particular 
\ref{general}$(3)$ yields that
$\codim\N((E_0)_s,X_s)=\codim\N(E_0,X)-s=c_X-s$, hence 
either $s=c_X$ and $\N((E_0)_s,X_s)=\N(X_s)$, or $s=c_X-1$ and
$\codim\N((E_0)_s,X_s)=1$. 
\begin{parg2}\label{infine}
Suppose that there exists an extremal ray $R$ 
of $X_s$ with $(E_0)_s\cdot R>0$ and $\Lo(R)\subsetneq X_s$. Then
$s=c_X-1$ and $R\not\subset\N((E_0)_s,X_s)$.

\medskip

Since we have shown that  $\N((E_0)_s,X_s)=\N(X_s)$ when $s=c_X$, it 
is enough to show that
$R\not\subset\N((E_0)_s,X_s)$.

We first show that $R\not\subset
\NE((E_0)_s,X_s)$. Otherwise, since $\NE((E_0)_s,X_s)\subseteq\NE(X_s)$, $R$
should be a one-dimensional face of $\NE((E_0)_s,X_s)$.
We have $\NE(E_0,X)=R_0+\NE(F_0,X)$ and 
$\NE((E_0)_s,X_s)=\sigma_*(R_0)+\sigma_*(\NE(F_0,X))$.
On the other hand \ref{tre}(5) and \eqref{sab3} give
$$\sigma_*(\NE(F_0,X))
\subset\sigma_*(\N(F_0,X))
=\sigma_*(L)\subseteq (E_0)_s^{\perp},$$ while
$(E_0)_s\cdot
R>0$, therefore we get $R=\sigma_*(R_0)$. But
$(E_0)_s$ is covered by the 
curves $\sigma(f_0)$, so that 
 $\Lo(R)\supseteq
(E_0)_{s}$, which is impossible. 

Therefore $R\not\subset
\NE((E_0)_s,X_s)$, and in particular
the contraction of $R$ is finite on $(E_0)_s$. Since $(E_0)_s\cdot
R>0$, this means that the contraction of $R$ has fibers of dimension
$\leq 1$, therefore $R$ is of type
$(n-1,n-2)^{sm}$ by \cite[Theorem 2.3]{ando} and \cite[Theorem 1.2]{wisn}.

In particular, $E_R:=\Lo(R)$ is a prime divisor covered by curves of
anticanonical degree $1$. Moreover these curves have class in $R$,
thus they cannot be contained in $T_1\cup\cdots\cup T_s$, because 
$T_1\cup\cdots\cup T_s\subset (E_0)_s$. By a standard argument (see
for instance \cite[Remark 2.3]{fanos}) we deduce that 
$E_R\cap
(T_1\cup\cdots\cup T_s)=\emptyset$,   
 hence by \eqref{sab0} and Remark \ref{elem} we have
 $$\sigma_*(L)=\N(T_1,X_s)\subseteq E_R^{\perp}.$$ 
Moreover $E_R\cdot \sigma(f_0)\geq
 0$, because $E_R\neq (E_0)_s$ (as $(E_0)_s\cdot R>0$).

We  show that  $R\not\subset\N((E_0)_s,X_s)$. By contradiction,
suppose that
 $R\subset\N((E_0)_s,X_s)$, and let $C$ be
 an irreducible curve with class in $R$. Then by \eqref{sab1} we have
$[C]=\lambda[\sigma(f_0)]+\gamma$, with $\lambda\in\R$ and
$\gamma\in
\sigma_*(L)$. Using \eqref{sab3} we get
$0<(E_0)_s\cdot C=\lambda (E_0)_s\cdot \sigma(f_0)$ and $(E_0)_s\cdot
\sigma(f_0)>0$,  thus
$\lambda>0$. On the other 
hand $-1=E_R\cdot C=\lambda E_R\cdot\sigma(f_0)$, which gives a
contradiction.  Thus $R\not\subset\N((E_0)_s,X_s)$.
\end{parg2}
\begin{parg2}\label{reduction}
We show that we can assume that there exists an extremal ray $R$ of
$X_s$ such that $(E_0)_s\cdot R>0$ and $\Lo(R)=X_s$.

\medskip

This is clear if $s=c_X$, by \ref{infine}. Suppose that
$s=c_X-1$, and consider an extremal ray $R$ of $X_s$ with
$(E_0)_{c_X-1}\cdot R>0$. If $\Lo(R)=X_{c_X-1}$, we are done; otherwise, by 
\ref{infine}, we have $R\not\subset\N((E_0)_{c_X-1},X_{c_X-1})$.

Let $\sigma_{c_X-1}\colon
X_{c_X-1}\to X_{c_X}$ be the contraction of $R$, and consider the sequence 
$$X\stackrel{\sigma}{\longrightarrow}
 X_{c_X-1}\stackrel{\sigma_{c_X-1}}{\longrightarrow}X_{c_X}.$$
Again, factoring $\sigma$ as a sequence of $c_X-1$ blow-ups, we can view 
this
as a part of a special Mori program for $-E_0$ in $X$, with $c_X$ steps,
and at each step $Q_i\not\subset\N((E_0)_i,X_i)$.

The $\pr^1$-bundles determined by this special Mori program are
$E_1,\dotsc,E_{c_X-1}$, and the transform of $E_R$ in $X$;  
the associated extremal rays (see Lemma~\ref{a}) are
$R_1,\dotsc,R_{c_X-1}$, and an additional extremal ray
$R_{c_X}$.

Since $E_1\cdot R_0>0$, Lemma~\ref{tre} still applies, thus we can
just replace
$R_1,\dotsc,R_{c_X-1}$ with
$R_1,\dotsc,R_{c_X}$, and restart. Since now the extremal rays are
$c_X$ (instead of $c_X-1$), we are done by what precedes.
\end{parg2}
\begin{parg2}
By \ref{reduction} there exists an elementary contraction of
fiber type
 $\ph\colon X_{s}\to Y$ 
 such that $(E_0)_{s}\cdot\NE(\ph)>0$;  set
$\psi:=\ph\circ\sigma\colon X\to Y$, and notice that $\ph((E_0)_s)=\psi(E_0)=Y$. 
$$\xymatrix{X \ar@/^1pc/[rr]^{\psi}\ar[r]_{\sigma}
  &{X_{s}}\ar[r]_{\ph}& Y}$$
The sequence above is a Mori program for $-E_0$, with $s$ steps,
and at each step $Q_i\not\subset\N((E_0)_i,X_i)$.
By \ref{pluto}$(2)$ we have two possibilities: either
$\N((E_0)_s,X_s)=\N(X_s)$ and $s=c_X$, or $\NE(\ph) 
\not\subset\N((E_0)_s,X_s)$ and $s=c_X-1$. 

Since $\N(T_1,X_s)\subseteq (E_0)_{s}^{\perp}$ by \eqref{sab0} and \eqref{sab3}, 
$\ph$ must be finite on
$T_1$, so that $\dim Y\geq n-2$.
\end{parg2} 
\begin{parg2}{\bf First case:  $\ph$ is not finite on $(E_0)_{s}$.}\label{316}
In this case 
 $\NE(\ph)\subset\N((E_0)_s,X_s)$, therefore
$\N((E_0)_s,X_s)=\N(X_s)$ and $s=c_X$. This also shows that
$L=E_0^{\perp}\cap\cdots\cap E_{c_X}^{\perp}$,
by \ref{tre}$(6)$.

Recall that $Y=\psi(E_0)$, and that
 $E_0\cong \pr^1\times F$ is smooth and Fano by \ref{tre}$(1)$. 
Moreover if $F_0:=\{pt\}\times F\subset E_0$, then $\N(F_0,X)=L$ by 
 \ref{tre}$(5)$. Finally $\N(\sigma(F_0),X_{c_X})=
\sigma_*(L)\subseteq (E_0)_{c_X}^{\perp}$ by  \eqref{sab3}, so that $\ph$ is 
finite on $\sigma(F_0)$. Since $\sigma$ is finite on $E_0$, we deduce 
that $\psi$ is finite on $F_0$.

Let 
$E_0\stackrel{\alpha}{\to}\widetilde{Y}\to Y$ be the Stein factorization
 of $\psi_{|E_0}$. Since   $\ph$ is not finite on $(E_0)_{c_X}$, 
$\alpha$ is a non-trivial contraction of $E_0$. On the other hand 
$\alpha$ is finite on $F_0$: the only possibility is that 
$\widetilde{Y}\cong F$ and $\alpha$ is the projection.

We deduce that
$\dim Y=n-2$ and that $\psi$ contracts $f_0$, hence
$\NE(\ph)=\sigma_*(R_0)$. Thus
$\NE(\psi)$ is a
$(c_X+1)$-dimensional face of $\NE(X)$ containing
$R_0,\dotsc,R_{c_X}$; in particular $\rho_Y=\rho_X-c_X-1$.  

Let us consider the divisor $$H:=2E_0+\sum_{i=1}^{c_X}E_i$$ on $X$. By
\ref{tre}$(4)$ we have 
$H\cdot R_i>0$ for every $i=0,\dotsc,c_X$, and
$$L= E_0^{\perp}\cap\cdots\cap E_{c_X}^{\perp}\subseteq H^{\perp}.$$ 

Recall from \ref{tre}$(1)$ and \ref{tre}$(5)$ that for every
$i=0,\dotsc,c_X$ we have $E_i\cong\pr^1\times F$, and if
$F_i:=\{pt\}\times F\subset E_i$, then $\N(F_i,X)=L\subset
H^{\perp}$. In particular $\NE(E_i,X)=R_i+\NE(F_i,X)\subset R_i+L$.

Let $C\subset X$ be an irreducible curve with 
$C\subset\Supp H=E_0\cup\cdots\cup E_{c_X}$. Then $C\subseteq E_i$ for
some $i\in\{0,\dotsc,c_X\}$, hence
$[C]\in R_i+L$ and $H\cdot C\geq 0$.

On the other hand, since $H$ is effective, we have $H\cdot C'\geq 0$
for every irreducible curve $C'$ not contained in $\Supp H$. 
Therefore $H$ is nef and defines a
contraction $\xi\colon X\to S$ such that $\NE(\xi)=H^{\perp}\cap\NE(X)$. 
$$\xymatrix{
& X\ar[dl]_{\xi}\ar[dr]^{\psi}\ar[r]^{\sigma} & {X_{c_X}}\ar[d]^{\ph} \\
S && Y
}$$
Let $i\in\{0,\dotsc,{c_X}\}$. Since $\N(F_i,X)\subset H^{\perp}$,
the image $\xi(F_i)$ is a point, and
$\xi(E_i)=\xi(f_i)$
is an irreducible rational curve (because $H\cdot f_i>0$). Therefore
 $\xi_{|E_i}\colon E_i\to \xi(f_i)$ factors through the
projection $E_i\to\pr^1$.
In particular $\dim\xi(\Supp H)=1$, hence $S$ is a surface
by \ref{divisors}$(1)$.

Let us show that
\renewcommand{\theequation}{\thethmdue}\stepcounter{thmdue}
\begin{equation}\label{disp}
\NE(\xi)=L\cap\NE(X).\end{equation}
We already have $\NE(\xi)=H^{\perp}\cap\NE(X)\supseteq
L\cap\NE(X)$. Conversely, let 
 $C_1\subset X$ be an irreducible curve such that $\xi(C_1)=\{pt\}$,
 \emph{i.e.}\ $H\cdot C_1=0$.

If $C_1$ is disjoint from
$\Supp H=E_0\cup \cdots\cup E_{c_X}$, then $C_1\cdot E_i=0$ for
$i=0,\dotsc,c_X$, hence $[C_1]\in L$.  

If instead  $C_1$ intersects $E_0\cup
\cdots\cup E_{c_X}$,  then it must be contained in it,
and we have $C_1\subset E_i$ for some $i$. Since $\xi_{|E_i}$ factors as
the projection onto $\pr^1$ followed by a finite map, we get $C_1\subset
F_i$, and again $[C_1]\in\N(F_i,X)=L$.
Therefore we have \eqref{disp}.

In particular, for every $i=0,\dotsc,c_X$ we have $\NE(\xi)\subseteq
E_i^{\perp}$, therefore  $E_i=\xi^*(\xi(E_i))$ by \ref{divisors}$(2)$.

\medskip

Let $\pi\colon X\to S\times Y$ be the morphism induced by $\xi$ and
$\psi$. We have $\ker\psi_*=\R(R_0+\cdots+R_{c_X})$, and
$\ker\psi_*\cap L=\{0\}$ by \ref{tre}$(5)$.  Moreover 
$\ker\xi_*\subseteq L$ by
\eqref{disp}, therefore $\pi$ is finite. 

In particular, $\xi$ must be equidimensional, hence $S$ is smooth by
\cite[Proposition 1.4.1]{ABW} and \cite[Lemma~3.10]{fanos}. We
need the following remark.
\begin{remark2}\label{prod}
Let $W$ be a smooth Fano variety and suppose we have two contractions
$$\xymatrix{ & W\ar[dl]_{\pi_1} \ar[dr]^{\pi_2} & \\
{W_1}&&{W_2} }$$
such that $W_1$ is smooth and
the induced morphism $\pi\colon W\to W_1\times W_2$ is
  finite. Consider the relative canonical divisor
  $K_{W/W_1}:=K_W-\pi_1^*K_{W_1}$. 
If $\ker(\pi_2)_*\subseteq (K_{W/W_1})^{\perp}$
in $\N(W)$,
then $\pi$ is an isomorphism.

This is rather standard, we give a proof for the reader's convenience.
Let $d$ be the degree of $\pi$, and $F\subset W$ a general fiber of
$\pi_2$; the restriction $f:=(\pi_1)_{|F}\colon F\to W_1$ is finite of degree
$d$. We observe that $F$ is Fano, hence numerical and linear
equivalence for divisors in $F$ coincide, 
 and by assumption $(K_{W/W_1})_{|F}\equiv 0$.
Then
$$K_F = (K_W)_{|F} = (\pi_1^*K_{W_1})_{|F} = f^*K_{W_1},$$
so that $f$ is \'etale. 
Therefore $W_1$ is Fano too, in particular it is simply connected,
thus $f$ is an 
isomorphism and $d = 1$.
\end{remark2}
We carry on with the proof of Proposition \ref{primameta}. We want to apply
 Remark \ref{prod} to deduce that $\pi\colon X\to S\times Y$ 
is an isomorphism;
for this
we just need to show that $K_{X/S}\cdot R_i=0$ for
$i=0,\dotsc,c_X$, because
$\ker\psi_*=\R(R_0+\cdots+R_{c_X})$. But
this follows easily because $E_i$ are products. 

Indeed since both $S$ and 
$E_i$ are smooth, \ref{divisors}$(4)$ yields that
$\xi(E_i)$ is a smooth curve.
Therefore $\xi(E_i)\cong\pr^1$ and $\xi_{|E_i}$ is the
projection, hence 
$$K_{X/S}\cdot f_i=(K_{X/S})_{|E_i}\cdot f_i=K_{E_i/\xi(E_i)}\cdot f_i=0.$$
Thus we conclude that $\pi$ is an isomorphism and $X\cong S\times Y$.
Moreover since $\rho_Y=\rho_X-c_X-1$, we have $\rho_S=c_X+1$. 
\end{parg2} 
\begin{parg2}{\bf Second case:  $\ph$ is finite on $(E_0)_s$.}\label{second}
Then $\dim Y=n-1$ and every fiber of $\ph$ is one-dimensional;
moreover every fiber of $\psi$ has an irreducible component of
dimension $1$.
Since $X$ and $X_s$ are Fano, \cite[Lemma~2.12 and Theorem 4.1]{AWaview} show that 
$Y$ is smooth and that
 $\ph$ and $\psi$ are conic bundles.
$$\xymatrix{X \ar@/^1pc/[rr]^{\psi}\ar[r]_{\sigma}
  &{X_{s}}\ar[r]_{\ph}& Y}$$
Set $Z_i:=\ph(T_i)=\psi(E_i)\subset Y$ for $i=1,\dotsc,s$. By standard
arguments on conic bundles (as at the end of the proof of
Lemma~\ref{conicbdl}), we see that  
$Z_1,\dotsc,Z_s\subset Y$ are pairwise disjoint smooth prime divisors, 
and that $\ph$ is smooth over
$Z_1\cup\cdots\cup Z_s$. For $i=1,\dotsc,s$ let $\widehat{E}_i\subset
X$ be the transform of $\ph^{-1}(Z_i)\subset X_s$, so that
$\psi^{-1}(Z_i)=E_i\cup \widehat{E}_i$. 
Then
$\widehat{E}_i$ is a smooth $\pr^1$-bundle with fiber
$\widehat{f}_i\subset\widehat{E}_i$, such that  $\widehat{E}_i\cdot
\widehat{f}_i=-1$. Moreover 
$f_i+ \widehat{f}_i$ is numerically equivalent to a general fiber of
$\psi$, and
$E_i\cdot \widehat{f}_i=\widehat{E}_i\cdot {f}_i=1$.

In particular, the divisors
$E_0,E_1,\dotsc,E_s,\widehat{E}_1,\dotsc,\widehat{E}_s$ are all
distinct (recall that $\psi(E_0)=Y$), and 
$E_1\cup\widehat{E}_1,\dotsc,E_s\cup\widehat{E}_s$ are pairwise disjoint.

Let us show that $[E_0], [E_1],\dotsc,[E_{s}],[\widehat{E}_1]$ are
linearly independent in $\mathcal{N}^1(X)$. Indeed suppose that 
$$aE_0+\sum_{i=1}^{s}b_iE_i+d\widehat{E}_1\equiv 0,$$
with $a,b_i,d\in\R$. Intersecting with a general fiber of $\psi\colon
X\to Y$, we get $a=0$. Intersecting with $f_2,\dotsc,f_{s}$, we get
$b_2=\cdots=b_{s}=0$. Finally intersecting with $f_1$ we get $d=b_1$,
that is, $d(E_1+\widehat{E}_1)\equiv 0$, which yields $d=0$, and we
are done. 

If $i,j\in\{1,\dotsc,s\}$ with $i\neq j$, we have
$E_i\cap \widehat{E}_j=\emptyset$, and hence
$L\subseteq\N(E_i,X)\subseteq \widehat{E}_j^{\perp}$ (see Remark
\ref{elem}).  Therefore by \ref{tre}$(6)$
$$
L\subseteq E_0^{\perp}\cap E_1^{\perp}\cap\cdots
\cap E_{s}^{\perp}\cap\widehat{E}_1^{\perp}\cap\cdots
\cap \widehat{E}_{s}^{\perp}\subseteq
E_0^{\perp}\cap E_1^{\perp}\cap\cdots
\cap E_{s}^{\perp}\cap\widehat{E}_1^{\perp}.
$$
Since the classes of $E_0,\dotsc,E_{s},\widehat{E}_1$ in
$\mathcal{N}^1(X)$
are linearly independent and $s\geq c_X-1$, we get
$$c_X+1=\codim L\geq s+2\geq c_X+1,$$
which yields $s=c_X-1$ and
$$L=E_0^{\perp}\cap E_1^{\perp}\cap\cdots
\cap E_{c_X-1}^{\perp}\cap\widehat{E}_1^{\perp}=
E_0^{\perp}\cap E_1^{\perp}\cap\cdots
\cap E_{c_X-1}^{\perp}\cap\widehat{E}_1^{\perp}\cap\cdots
\cap \widehat{E}_{c_X-1}^{\perp}.$$

Let $i\in\{1,\dotsc,c_X-1\}$.
Observe that
$[\widehat{f}_i]\not\in\N(E_i,X)$: otherwise by \ref{tre}$(5)$
we would
have $\widehat{f}_i\equiv \lambda f_i+\gamma$, with $\lambda\in\R$ and
$\gamma\in L\subset E_0^{\perp}\cap E_i^{\perp}$. 
Intersecting with $E_i$ we get $\lambda=-1$, hence
$E_0\cdot \widehat{f}_i=-E_0\cdot {f}_i<0$, which is impossible
because $E_0\neq  \widehat{E}_i$. 
We also notice that $E_0$ cannot contain any curve $\widehat{f}_i$,
because $\sigma(\widehat{f}_i)$ is a fiber of $\ph$, and
 $\ph$ is finite on $(E_0)_{c_X-1}=\sigma(E_0)$.

Therefore we can apply Lemma~\ref{ultimissimo} to $E_0$ and
$E_1,\dotsc,E_{c_X-1},\widehat{E}_1,\dotsc,\widehat{E}_{c_X-1}$, 
and we get:
$$\codim\N(\widehat{E}_i,X)=c_X\ \text{ and
}\ R_i\not\subset\N(\widehat{E}_i,X)\ \text{ for every $i=1,\dotsc,c_X-1$.}$$

Fix again $i\in\{1,\dotsc,c_X-1\}$.
 Lemma~\ref{expl}, applied to $\widehat{E}_i$ and $R_i$,
yields that
$$\N(E_i\cap \widehat{E}_i,X)=\N(\widehat{E}_i,X)\cap E_i^{\perp}
=\N(E_i,X)\cap E_i^{\perp}=L$$
(see \eqref{utile} for the last equality).
Finally
we apply Lemma~\ref{blah} to $D=E_i$ and
$E=\widehat{E}_i$, and we deduce that $\widehat{R}_i:=\R_{\geq
  0}[\widehat{f}_i]$ is an extremal ray of 
type $(n-1,n-2)^{sm}$, $\widehat{E}_i\cong\pr^1\times
\widehat{F}^i$, and $E_i\cap \widehat{E}_i=\{pts\}\times
\widehat{F}^i\subset \widehat{E}_i$. 
On the other hand again  Lemma~\ref{blah}, applied now to 
$D=\widehat{E}_i$ and $E=E_i$, shows that
$E_i\cap \widehat{E}_i=\{pts\}\times
F\subset E_i\cong\pr^1\times F$, hence $\widehat{F}^i=F$.

Observe that $\NE(\psi)=R_1+\widehat{R}_1+\cdots+R_{c_X-1}
+\widehat{R}_{c_X-1}$
has dimension $c_X$, and that
 $\psi_{|E_0}\colon
E_0\cong\pr^1\times F_0 \to Y$ is finite. We need the following lemma.
\begin{lemma2}\label{Vequivalence}
Let $E$ be a projective manifold and
$\pi\colon E\to W$ a $\pr^1$-bundle with fiber $f\subset E$. Moreover let 
$\psi_0\colon E\to Y$ be a morphism onto a projective
manifold $Y$, such that $\dim\psi_0(f)=1$. Suppose that there exists a
prime divisor  
$Z_1\subset Y$ such that $\N(Z_1,Y)\subsetneq\N(Y)$ and
$\psi_0^*(Z_1)\cdot f>0$. Then there is a commutative diagram:
$$\xymatrix{E\ar[r]^{\psi_0}\ar[d]_{\pi}& Y\ar[d]^{\zeta}\\
W\ar[r]&{Y'}
}$$ 
where $Y'$ is smooth and $\zeta$ is a smooth morphism with fibers
isomorphic to $\pr^1$. 
\end{lemma2}
\begin{proof}[Proof of Lemma \ref{Vequivalence}]
Consider the morphism $\phi\colon E\to W\times Y$ induced by $\pi$ and $\psi_0$,
set $E':=\phi(E)\subset  W\times Y$, and let $\pi'\colon E'\to W$ 
be the projection. For every $p\in W$ we have
$\pi^{-1}(p)=\phi^{-1}((\pi')^{-1}(p))$, hence  
$(\pi')^{-1}(p)=\psi_0(\pi^{-1}(p))\subset Y$ is an
irreducible and reduced rational curve in $Y$.

Now $\pi'\colon E'\to W$ is a well defined family of algebraic
one-cycles on $Y$ over $W$ (see \cite[Definition I.3.11 and Theorem
  I.3.17]{kollar}), and 
induces a morphism $\iota\colon W\to\Chow(Y)$. Set
$V:=\iota(W)\subset\Chow(Y)$. Then $V$ is a proper, covering family of
irreducible and reduced rational curves on $Y$,
so that  $V$ is an \emph{unsplit} family (see
\cite[Definition IV.2.1]{kollar}).

The family $V$ induces an equivalence relation on $Y$ as a set, called
$V$-equivalence; we refer the reader to \cite[\S5]{debarreUT} and
references therein for the related definitions and properties.

We have $Z_1\cdot\psi_0(f)>0$; in particular $Z_1$ intersects every
 $V$-equivalence class in $Y$. This implies that 
$$\N(Y)=\R[\psi_0(f)]+\N(Z_1,Y)$$
(see for instance \cite[Lemma~3.2]{occhetta}).
On the other hand by assumption
$\N(Z_1,Y)\subsetneq\N(Y)$,
therefore
$[\psi_0(f)]\not\in\N(Z_1,Y)$.

Let $T\subseteq Y$ be a $V$-equivalence class; notice that $T$ is either a
closed subset, or 
a countable union of closed subsets.
Let
 $T_1\subseteq T$ be an irreducible closed subset with $\dim T_1=\dim T$. 
We have $\N(T_1,Y)=\R[\psi_0(f)]$ by
\cite[Proposition IV.3.13.3]{kollar}, and $T_1\cap Z_1\neq\emptyset$. This
implies that $\dim  (T_1\cap Z_1)=0$ and $\dim T=\dim T_1=1$, that is:
\emph{every $V$-equivalence class has dimension $1$}. 
Then by
\cite[Proposition 1]{unsplit} there exists a contraction
$\zeta\colon
Y\to Y'$ whose fibers coincide with $V$-equivalence classes. 

Since $Y$ is smooth, $Y'$ is irreducible, and $\zeta$ has connected fibers, 
the
general fiber $l$ of $\zeta$ is irreducible and smooth. 
But $l$ is a $V$-equivalence class and $\dim l=1$, hence
$l\cong\pr^1$ and $-K_Y\cdot l=2$. 
Moreover
 $\NE(\zeta)=\R_{\geq 0}[l]$, so $-K_Y$ is $\zeta$-ample;
this implies that 
 $\zeta$ is an elementary contraction and a conic bundle, and
 that $Y'$ is smooth (see \cite[Theorem 3.1]{ando}). 
Finally $\zeta$ 
cannot have singular fibers, because the family $V$ is unsplit.
\end{proof}
Let us carry on with the proof of Proposition \ref{primameta}.
We have $\psi^*(Z_1)\cdot f_0=(E_1+\widehat{E}_1)\cdot f_0>0$,
and $\N(Z_1,Y)\subseteq Z_2^{\perp}\subsetneq\N(Y)$ because $Z_1\cap
Z_2=\emptyset$
(see Remark \ref{elem}).
Therefore we can apply Lemma~\ref{Vequivalence} to $E_0$ and
$\psi_0:=(\psi)_{|E_0}\colon E_0\to Y$. This shows
that $[\psi(f_0)]$ belongs to an extremal ray of $Y$,
whose contraction is a smooth conic bundle
$\zeta\colon Y\to Y'$.

We consider the composition
 $\psi':=\zeta\circ\psi\colon X\to Y'$; the cone $\NE(\psi')$ is a
 $(c_X+1)$-dimensional face of $\NE(X)$ containing
 $R_0,R_1,\dotsc,R_{c_X-1},\widehat{R}_1,\dotsc,\widehat{R}_{c_X-1}$,
and $\rho_{Y'}=\rho_X-c_X-1$. 

Now we proceed similarly to the previous case. 
   Let us consider the divisor
$$H':=2E_0+2\sum_{i=1}^{c_X-1}E_i+\sum_{i=1}^{c_X-1}\widehat{E}_i$$ 
on $X$. We have $H'\cdot R_0>0$,
$H'\cdot R_i>0$ and $H'\cdot \widehat{R}_i>0$ 
for every $i=1,\dotsc,c_X-1$, and $(H')^{\perp}\supseteq L$. As before, $H'$
is nef and defines a 
contraction onto a surface $\xi'\colon X\to S$, such that 
$\xi'(E_0)$, $\xi'(E_i)$, and $\xi'(\widehat{E}_i)$
are irreducible rational curves and $E_0=(\xi')^*(\xi'(E_0))$,
$E_i=(\xi')^*(\xi'(E_i))$, $\widehat{E}_i=(\xi')^*(\xi'(\widehat{E}_i))$
for all $i=1,\dotsc,c_X-1$.
$$\xymatrix{
& X\ar[d]^{\psi'}\ar[dl]_{\xi'}\ar[dr]^{\psi}\ar[r]^{\sigma} &
  {X_{c_X-1}}\ar[d]^{\ph} \\ 
{S} & {Y'} & Y\ar[l]^{\zeta}
}$$

Then we consider the morphism $\pi'\colon X\to S\times Y'$ induced by
$\xi'$ and $\psi'$. 
As in the previous case, one sees first that $\pi'$ is
finite, and then that it is an isomorphism, applying Remark \ref{prod}.
Finally we have $\rho_{S}=c_X+1$, because $\rho_{Y'}=\rho_X-c_X-1$. 
\end{parg2}
\begin{parg2}
We have shown in \ref{316} and \ref{second} that $X\cong S\times T$,
where $S$ is a Del Pezzo surface with $\rho_{S}=c_X+1$ 
(and $T=Y$ in \ref{316}, while $T=Y'$ in \ref{second}).
In particular
$c_X\leq 8$, as $\rho_S\leq 9$. Finally
 $c_T\leq c_X$ by Example \ref{max}, and
this concludes the proof of Proposition \ref{primameta}.
\end{parg2}
\vspace{-6pt}
\end{proof}
\subsection{The case of codimension $3$}\label{sectiontre}
In this section we show the following.
\begin{proposition2}\label{secondameta}
Let $X$ be a Fano manifold with $c_X=3$.
Then there exists a flat, quasi-elementary
contraction $X\to T$ where $T$ is an $(n-2)$-dimensional Fano
manifold, $\rho_X-\rho_T=4$, and $c_T\leq 3$.
\end{proposition2}
\begin{proof}
By Corollary \ref{asilo}, there exist
 a prime divisor $D\subset X$
  with $\codim\N(D,X)=3$, and a special Mori program for $-D$, such
  that $Q_k\not\subset\N(D_k,X_k)$.  
\renewcommand{\theequation}{\thethmdue}\stepcounter{thmdue}
\begin{equation}
\label{veronica}
\xymatrix{
{X=X_0}\ar@/^1pc/@{-->}[rrrr]^{\sigma}\ar@/_1pc/@{-->}[drrrr]_{\psi}
\ar@{-->}[r]_{\,\quad \sigma_0}& 
{X_1}\ar@{-->}[r]&
{\cdots}\ar@{-->}[r] &
{X_{k-1}}\ar@{-->}[r]_{\sigma_{k-1}}& 
{X_k}\ar[d]^{\ph}\\
&&&&Y} 
\end{equation}
We apply Lemmas \ref{pluto} and \ref{conicbdl}. By \ref{pluto}(2) and
\ref{pluto}(3),  there exist exactly two 
indices $i_1,i_2\in\{0,\dotsc,k-1\}$ such that
$Q_{i_j}\not\subset\N(D_{i_j},X_{i_j})$; the $\pr^1$-bundles
$E_1,E_2\subset X$ determined by the Mori program are the transforms
of $\Exc(\sigma_{i_1}),\Exc(\sigma_{i_2})$ respectively. Let moreover
$\widehat{E}_1,\widehat{E}_2\subset X$ be as in \ref{conicbdl}(4).  
Recall that for $i=1,2$ $E_i$ (respectively, $\widehat{E}_i$) is a
smooth $\pr^1$-bundle with fiber $f_i\subset E_i$ (respectively,
$\widehat{f}_i\subset\widehat{E}_i$), such that $E_i\cdot
f_i=\widehat{E}_i\cdot\widehat{f}_i=-1$, $E_i\cdot\widehat{f}_i>0$,
and $\widehat{E}_i\cdot f_i>0$. Moreover $(E_1\cup\widehat{E}_1)\cap  
(E_2\cup\widehat{E}_2)=\emptyset$.  
\begin{parg2}\label{terzooutline}
Before going on, let us
 give an outline of what we are going to do.
 
Our goal is
to show that $k=2$ and
$\sigma$ is just the composition of two smooth blow-ups with
exceptional divisors $E_1$ and $E_2$. The proof of this fact
is quite technical, and
 will be achieved in several
steps.

We first show in~\ref{details} some properties of $\N(E_i,X)$ and
$\N(\widehat{E}_i,X)$
which are needed in the sequel.

In \ref{F} we prove that if  $F\subset X$ is 
a prime divisor whose class
in $\mathcal{N}^1(X)$ spans a one-dimensional face of the cone of
effective divisors $\Eff(X)\subset\mathcal{N}^1(X)$ (see \ref{324}), then
$F$ must intersect both
$E_1\cup \widehat{E}_1$ and 
$E_2\cup
\widehat{E}_2$. 

Then  we  show in \ref{325} that  the Mori program
\eqref{veronica} contains only two divisorial contractions,
the ones with exceptional divisors $E_1$ and $E_2$. We proceed by
contradiction, applying \ref{F} to the exceptional divisor of a
divisorial contraction (different from $\sigma_{i_1}$ and
$\sigma_{i_2}$)
in the Mori program.

In \ref{326b} and \ref{327} we prove the existence of two disjoint
prime divisors $F,\widehat{F}\subset X$, which are smooth
$\pr^1$-bundles with fibers $l\subset F$,
$\widehat{l}\subset\widehat{F}$ such that $F\cdot
l=\widehat{F}\cdot\widehat{l}=-1$, which are horizontal for the
rational conic bundle $\psi\colon X\dasharrow Y$, 
and intersect the divisors $E_1,
E_2,\widehat{E}_1,\widehat{E}_2$ in a suitable way. 

Finally in \ref{328} and \ref{329} we use
$F$ and $\widehat{F}$ to show that 
the Mori program \eqref{veronica} contains no flips.
This means that $k=2$, $X_2$ and $Y$ are smooth,
$\sigma$ is just a smooth blow-up  with exceptional divisors
$E_1$ and $E_2$, and $\ph$ and $\psi$ are
conic bundles. 

The situation is now analogous to the one in \ref{second}, 
and similarly to that case
we prove that there is a smooth conic bundle $Y\to Y'$,
where $\dim Y'=n-2$ (see \ref{330}). We have $\rho_X-\rho_{Y'}=4$, and
the contraction $X\to Y'$ is flat
and quasi-elementary. 

To conclude, in \ref{331} 
we show that the conic bundle $\ph\colon X_2\to Y$
is smooth. This implies that every fiber of
the conic bundle $\psi\colon X\to Y$
is reduced, and hence by a result in \cite{wisn}
both $Y$ and $Y'$ are Fano.
\end{parg2}
\begin{parg2}\label{details}
For $i=1,2$ we have:
$$\codim\N(E_i,X)=\codim\N(\widehat{E}_i,X)=3,\quad
[\widehat{f}_i]\not\in\N({E}_i,X),\ 
\text{ and }\ 
[f_i]\not\in\N(\widehat{E}_i,X);$$
in particular $\N(E_i,X)\neq\N(\widehat{E}_i,X)$.

Indeed $[\widehat{f}_i]\not\in\N(E_i,X)$ by
\ref{conicbdl}$(4)$. Moreover $D$ cannot contain any curve
$\widehat{f}_i$, because $\sigma(\widehat{f}_i)$ is a fiber of $\ph$, and
$\ph$ is finite on $D_{k}\subset X_k$.  
Therefore Lemma~\ref{ultimissimo} yields the statement.
\end{parg2}
\begin{parg2}\label{324}
Let $Z$ be a Mori dream space, and
$\Eff(Z)\subset\mathcal{N}^1(Z)$ the convex cone spanned by
classes of effective divisors. By
\cite[Proposition 1.11(2)]{hukeel} $\Eff(Z)$ is a closed, convex polyhedral cone.
If $F\subset Z$ is a prime divisor covered by a family of curves with
which $F$ has negative intersection, then it is easy to see that 
$[F]\in\mathcal{N}^1(Z)$ 
spans a one-dimensional face of $\Eff(Z)$, and that the
only prime divisor whose class belongs to this face is $F$
itself.  In
particular, this is true for $E_1,
E_2,\widehat{E}_1,\widehat{E}_2\subset X$ (recall that $X$ is a Mori
dream space by Theorem \ref{FF}).
\end{parg2}
\begin{parg2}\label{F}
Consider a prime divisor $F\subset X$ such that $[F]$
 spans a one-dimensional face of $\Eff(X)$. We show that 
if $F$ is different from $E_1,
E_2,\widehat{E}_1,\widehat{E}_2$, then  $F$ must intersect both
$E_1\cup \widehat{E}_1$ and 
$E_2\cup
\widehat{E}_2$. 

\medskip

 Indeed if for instance $F$ is disjoint from $E_1\cup
\widehat{E}_1$, then $\N(E_1,X)\cup\N(\widehat{E}_1,X)\subseteq
E_2^{\perp}\cap \widehat{E}_2^{\perp}\cap F^{\perp}$ (see
Remark \ref{elem}). However this is 
impossible, because since $[E_2],[\widehat{E}_2],[F]\in\mathcal{N}^1(X)$
span three distinct one-dimensional faces of $\Eff(X)$, they must be
 linearly independent, thus $E_2^{\perp}\cap \widehat{E}_2^{\perp}\cap
 F^{\perp}$ has codimension $3$, while
$\N(E_1,X)$ and $\N(\widehat{E}_1,X)$ are distinct subspaces of codimension
$3$ by \ref{details}.
\end{parg2}
\begin{parg2}\label{325}
Let us show that $\sigma_i$ is a flip
for every $i\in\{0,\dotsc,k-1\}\smallsetminus\{i_1,i_2\}$, namely that
$\sigma_{i_1}$ and $\sigma_{i_2}$ are the unique divisorial
contractions in the Mori program \eqref{veronica}. 

\medskip

 By contradiction, suppose that there exists
 $i\in\{0,\dotsc,k-1\}\smallsetminus\{i_1,i_2\}$ such that $\sigma_i$
is a divisorial contraction. By \ref{324} 
$\Exc(\sigma_i)\subset X_i$ is a prime divisor
whose class spans a one-dimensional face of $\Eff(X_i)$, and it is the
unique prime divisor in $X_i$ with class in $\R_{\geq
  0}[\Exc(\sigma_i)]$.\footnote{Notice that $X_i$ is again a Mori
  dream space.} 

 Let $G\subset X$ be
the transform of $\Exc(\sigma_i)$. By \ref{conicbdl}(3) and
\ref{conicbdl}(4) there exists an open subset $U\subseteq X$,
containing $E_1,
E_2,\widehat{E}_1,\widehat{E}_2$, such that $\sigma$ is regular on
$U$, and  $\Exc(\sigma_i)$ is disjoint from the image of $U$ in
$X_i$. Therefore $G\cap U=\emptyset$, in particular
the divisor  $G$ is 
disjoint from  $E_1,
E_2,\widehat{E}_1,\widehat{E}_2$.

 Then \ref{F} shows that
$[G]\in\mathcal{N}^1(X)$ cannot span
 an extremal ray of $\Eff(X)$. This means that $[G]=\sum_j\lambda_j [G_j]$
 with $\lambda_j\in\R_{>0}$ and $G_j\subset X$
 prime divisors such that $[G]\not\in\R_{\geq 0}[G_j]$; in particular $G_j\neq G$.

On the other hand, 
the map $\xi:=\sigma_{i-1}\circ\cdots\circ\sigma_0\colon X\dasharrow
X_i$ 
induces a surjective linear map 
$\xi_*\colon\mathcal{N}^1(X)\to\mathcal{N}^1(X_{i})$  
such
that $\xi_*(\Eff(X))=\Eff(X_i)$. 
Then in $\mathcal{N}^1(X_i)$ we get 
$$[\Exc(\sigma_i)]=[\xi_*(G)]=\sum_j\lambda_j[\xi_*(G_j)],$$  
hence $[\xi_*(G_j)]\in\R_{\geq 0}[\Exc(\sigma_i)]$ for every $j$. 
If $\xi_*(G_j)\neq 0$ for some $j$, then $\xi_*(G_j)$ is a prime divisor, and
we get $\xi_*(G_j)=\Exc(\sigma_i)$ and
hence $G_j=G$, a contradiction. 
Thus $\xi_*(G_j)=0$
for every $j$, therefore $[\Exc(\sigma_i)]=0$,  again a contradiction.
\end{parg2}
\begin{parg2}\label{326}
Let $F\subset X$ be a smooth prime divisor which is a
 $\pr^1$-bundle with $F\cdot l=-1$,
where 
$l\subset F$ is a fiber. Suppose that $F$ is
 different from $E_1,
E_2,\widehat{E}_1,\widehat{E}_2$.
Then:
\begin{enumerate}[$\bullet$]
\item
$F$ must intersect both $E_1\cup\widehat{E}_1$ and $E_2\cup\widehat{E}_2$; 
\item
 either $E_1\cdot l=\widehat{E}_1\cdot l= 
E_2\cdot l=\widehat{E}_2\cdot l=0$, or
$(E_1+\widehat{E}_1)\cdot l>0$ and 
$(E_2+\widehat{E}_2)\cdot l>0$.
\end{enumerate}

\medskip

By \ref{324} $[F]$ spans a one-dimensional face of $\Eff(X)$,
so that \ref{F} gives the first statement.  

Recall that $(E_1\cup\widehat{E}_1)\cap(E_2\cup\widehat{E}_2)=\emptyset$. 
If $(E_1+\widehat{E}_1)\cdot l=0$, since $F$ intersects $E_1\cup\widehat{E}_1$, 
there exists a fiber $\overline{l}$ of the $\pr^1$-bundle structure of
$F$ which is contained in $E_1\cup\widehat{E}_1$. Thus
$\overline{l}\cap (E_2\cup\widehat{E}_2)=\emptyset$, and we get
$(E_2+\widehat{E}_2)\cdot l=0$.  
In this way we see that the intersections
$(E_1+\widehat{E}_1)\cdot l$, $(E_2+\widehat{E}_2)\cdot l$ are either
both zero or both  positive, and this gives the second statement.
\end{parg2}
\begin{parg2}\label{326b}
We show that  there exist two disjoint smooth prime divisors
$F,\widehat{F}\subset X$, different from $E_1,
E_2,\widehat{E}_1,\widehat{E}_2$,
 such that:
\begin{enumerate}[$\bullet$]
\item $F$ and $\widehat{F}$ are $\pr^1$-bundles, with fibers 
$l\subset F$ and $\widehat{l}\subset \widehat{F}$ respectively, such that 
$F\cdot l=\widehat{F}\cdot \widehat{l}=-1$;
\item
 the intersections
$(E_1+\widehat{E}_1)\cdot l$, $(E_1+\widehat{E}_1)\cdot
\widehat{l}$,  
$(E_2+\widehat{E}_2)\cdot l$, $(E_2+\widehat{E}_2)\cdot
\widehat{l}$
are all positive.
\end{enumerate}

\medskip

 We have $\codim\N(E_1,X)=3$ (see \ref{details}).
Consider a special Mori program for $-E_1$ (which exists by
Proposition \ref{tomm}), and let  $G_1,\dotsc,G_s\subset X$ be the
$\pr^1$-bundles determined by the Mori program. Recall from
Lemma~\ref{pluto} that   $G_1,\dotsc,G_s$ are  
 pairwise disjoint smooth prime divisors, with
$2\leq s\leq 3$, 
such that
 every $G_i$ is a $\pr^1$-bundle with  $G_i\cdot
r_i=-1$, where  $r_i\subset G_i$ is a fiber;
moreover $E_1\cdot r_i>0$.
 In particular $G_i\neq E_1$ and
$G_i\cap E_1\neq\emptyset$, thus $G_i\neq E_2$ and $G_i\neq\widehat{E}_2$. 
Finally, if $G_i\neq\widehat{E}_1$,  by  \ref{326}
we have $(E_1+\widehat{E}_1)\cdot r_i>0$ and  
$(E_2+\widehat{E}_2)\cdot r_i>0$. 

Suppose that $\{G_1,\dotsc,G_s\}$ contains at least two divisors
distinct from 
 $\widehat{E}_1$, say $G_1$ and $G_2$. Then
 we set $F:=G_1$ and $\widehat{F}:=G_2$, and we are done.

Otherwise, we
have $s=2$ and $G_2=\widehat{E}_1$. Then Lemma~\ref{conicbdl} applies, and by
\ref{conicbdl}(4) there exists a
smooth prime divisor $\widehat{G}_2$, having a $\pr^1$-bundle
structure with fiber $\widehat{r}_2$, such that:  
$$\widehat{G}_2\cdot \widehat{r}_2=-1,\quad 
G_1\cap\widehat{G}_2=\emptyset,\quad \widehat{G}_2\neq E_1,\ \text{ and }\
\widehat{E}_1\cdot  \widehat{r}_2 =1.$$ 
 In particular $\widehat{G}_2\neq\widehat{E}_1$ and 
$\widehat{G}_2\cap\widehat{E}_1\neq\emptyset$, therefore
$\widehat{G}_2\neq {E}_2$ and $\widehat{G}_2\neq\widehat{E}_2$. By
\ref{326} 
we have $(E_1+\widehat{E}_1)\cdot \widehat{r}_2>0$ and  
$(E_2+\widehat{E}_2)\cdot \widehat{r}_2>0$, 
thus we set $F:=G_1$ and $\widehat{F}:=\widehat{G}_2$.
\end{parg2}
\begin{parg2}\label{327}
As soon as $F$ (respectively $\widehat{F}$)
intersects one of the divisors $E_i$, then
$F\cdot f_i>0$ and $E_i\cdot l>0$
(respectively $\widehat{F}\cdot f_i>0$ and $E_i\cdot \widehat{l}>0$),
and  similarly for $\widehat{E}_i$. In particular we have $F\cdot f>0$
and $\widehat{F}\cdot f>0$,  
where $f$ is a general fiber of $\psi$.

\medskip

 Suppose  for instance
 that $F\cap E_1\neq \emptyset$.
If $E_1\cdot l=0$, then $E_1$ contains some curve $l$, but this is
impossible because $(E_2+\widehat{E}_2)\cdot l>0$
while $E_1\cap (E_2\cup \widehat{E}_2)=\emptyset$; thus $E_1\cdot l>0$.

If $F\cdot f_1=0$, then $F$ contains an irreducible curve
$\overline{f}_1$ which is a fiber of the $\pr^1$-bundle structure on
$E_1$. 
Let $\pi\colon F\to G$ be the $\pr^1$-bundle structure on $F$, and
$\pi_*\colon\N(F)\to\N(G)$ the push-forward.  
Notice that $\pi(\overline{f}_1)$ is a curve, because $\overline{f}_1$
and $l$ are not numerically equivalent in $X$, and hence neither in
$F$.  

Consider the surface
 $S:=
\pi^{-1}(\pi(\overline{f}_1))$. Then
$\pi_*(\N(S,F))=\R\pi_*([\overline{f}_1]_F)$, hence
$\N(S,F)=\ker\pi_*\oplus
\R[\overline{f}_1]_F=\R[l]_F\oplus\R[\overline{f}_1]_F$, and $\N(S,X)=
\R[l]\oplus\R[{f}_1]$. 

Since $\widehat{E}_1\cdot
\overline{f}_1>0$, we have
$S\cap\widehat{E}_1\neq\emptyset$, and there exists an irreducible
curve $C\subset S\cap\widehat{E}_1$.  Thus $[C]\in\N(S,X)$, so that
$C\equiv\lambda l+\mu f_1$ with $\lambda,\mu\in\R$. On the other
hand $C\cap (E_2\cup \widehat{E}_2)=\emptyset$ (because
$C\subset\widehat{E}_1$) and 
$$0=(E_2+\widehat{E}_2)\cdot C=\lambda (E_2+\widehat{E}_2)\cdot 
l,$$
which by \ref{326b} yields 
$\lambda=0$, $\mu\neq 0$ and $[f_1]=(1/\mu)[C]\in\N(\widehat{E}_1,X)$,
a contradiction with \ref{details}.  

Therefore $F\cdot f_1>0$. We have $f\equiv f_1+\widehat{f}_1$ (see
\ref{conicbdl}$(4)$), and $F\cdot \widehat{f}_1\geq 0$ because
$F\neq\widehat{E}_1$ (see \ref{326b}), hence $F\cdot f>0$. 
\end{parg2}
\begin{parg2}\label{328}
For every $i\in\{0,\dotsc,k\}$ let $F_i,\widehat{F}_i\subset X_i$ be
the transforms of $F,\widehat{F}$. 
Let us show that for any 
$i\in\{0,\dotsc,k-1\}\smallsetminus\{i_1,i_2\}$,
the divisors $F_i$ and $\widehat{F}_i$ are disjoint from
$\Lo(Q_i)$.

\medskip

 By contradiction, suppose for instance that this is not true for $F$,
and let $j\in\{0,\dotsc,k-1\}\smallsetminus\{i_1,i_2\}$ be the
smallest index such that 
 $F_j$ intersects $\Lo(Q_j)$.
Recall from \ref{325} that $\sigma_i$ is a flip
for every $i\in\{0,\dotsc,k-1\}\smallsetminus\{i_1,i_2\}$; in
particular, $Q_j$ is a small extremal ray, and $\sigma_j$ is a flip. 

After \ref{conicbdl}(3),
$\sigma$ is regular on the divisors
$E_1,E_2,\widehat{E}_1,\widehat{E}_2$, 
and
$\Lo(Q_j)$ is disjoint from their images in $X_j$.

Recall from \ref{general}$(4)$ the definition of $A_i\subset X_i$, 
for $i\in\{1,\dotsc,k\}$: 
$A_1\subset X_1$ is the indeterminacy locus of $\sigma_0^{-1}$,
and for $i\geq 2$, if  $\sigma_{i-1}$ is a divisorial contraction
(respectively, if $\sigma_{i-1}$ is a flip), 
$A_i$ is the union of $\sigma_{i-1}(A_{i-1})$ (respectively, the
transform of $A_{i-1}$) and  
the indeterminacy locus of
$\sigma_{i-1}^{-1}$.

If $j>0$, by the minimality of $j$, $F_j$ does not intersect the loci of the
previous flips, hence it can intersect $A_j$ only along the images of
$E_1$ and 
$E_2$. Therefore 
\renewcommand{\theequation}{\thethmdue}\stepcounter{thmdue}
\begin{equation}\label{forno} 
\Lo(Q_j)\cap F_j\cap A_j=\emptyset.\end{equation} 

Let $\alpha_j\colon X_j\to Y_j$ be the contraction of $Q_j$. 
Suppose first that
$\alpha_j$ is finite on $F_j$. Then
$\Lo(Q_j)=\Exc(\alpha_j)\not\subset F_j$, and since $F_j\cap
\Lo(Q_j)\neq\emptyset$, we have 
 $F_j\cdot Q_j>0$. Hence
 every non
trivial fiber of $\alpha_j$ must have dimension $1$, otherwise
$\alpha_j$ would not be finite on $F_j$. 

If $j=0$, then $\alpha_0$ is a small contraction of a smooth Fano variety with
one-dimensional fibers, which is impossible, see  \cite[Theorem 4.1]{AWaview}.

Suppose that $j>0$.
If $C_0\subset X_j$ is an irreducible curve in a fiber of $\alpha_j$,
then $C_0$ 
must intersect $F_j$, hence $C_0\not\subseteq A_j$ by \eqref{forno}; 
in particular
$C_0\not\subseteq\Sing(X_j)$ (recall that $\Sing(X_j)\subseteq A_j$ by
\ref{general}$(4)$).   
 Then \cite[Lemma~1]{ishii} yields $-K_{X_j}\cdot C_0\leq 1$, and
\cite[Lemma~3.8]{31} implies that $C_0\cap A_j=\emptyset$.
We conclude that
$\Lo(Q_j)\subseteq X_j\smallsetminus A_j$, and this is again
impossible
by  \cite[Theorem 4.1]{AWaview},
because $-K_{X_j}\cdot Q_j>0$ and $(\alpha_j)_{|X_j\smallsetminus A_j}\colon
X_j\smallsetminus A_j\to Y_j\smallsetminus\alpha_j(A_j)$ is 
a small contraction of a smooth variety with
one-dimensional fibers.

\medskip

Suppose now that  $\alpha_j$ is not finite on $F_j$. Then
 there exists  an irreducible curve $C_1\subset F_j$ with
$[C_1]\in Q_j$; in particular $C_1$ is disjoint from the images of 
$E_1,
E_2,\widehat{E}_1,\widehat{E}_2$ in $X_j$.
 Consider the transform $\widetilde{C}_1\subset
F\subset X$ of $C_1$, so that $\widetilde{C}_1$ is disjoint from
$E_1,
E_2,\widehat{E}_1,\widehat{E}_2$.

Recall that $F$ intersects both $E_1\cup\widehat{E}_1$ and
$E_2\cup\widehat{E}_2$ by \ref{326}.
 We assume that $F$ intersects $E_1$ and $E_2$, the other cases being
 analogous. Then $E_1\cdot l>0$ by \ref{327}, 
so that using \ref{elem2}$(3)$ we get
$$\widetilde{C}_1\equiv \lambda l + \mu C_2,$$
where $C_2\subset F\cap E_1$ is a curve, $\lambda,\mu\in\R$, and
$\mu\geq 0$. In particular $C_2\cap E_2=\emptyset$, therefore 
$0=E_2\cdot \widetilde{C}_1= \lambda E_2\cdot l$. On the other hand 
 $E_2\cdot l>0$ by \ref{327}, and
this implies that 
$\lambda=0$ and $\widetilde{C}_1\equiv\mu C_2$. Recall that the
map $X\dasharrow X_j$ is regular on $F$ by the minimality of $j$, and
call $C_2'$ the image 
of $C_2$ in $X_j$.  
We deduce that $C_1\equiv\mu C_2'$ in $X_j$, so that $[C_2']\in
Q_j$. But $C_2'$ is contained in the image of $E_1$, which is disjoint
from  
$\Lo(Q_j)$, and 
 we have a contradiction.
\end{parg2}
\begin{parg2}\label{329}
We show that $k=2$ in \eqref{veronica}, so that $i_1=0$ and $i_2=1$.

\medskip

 By contradiction, suppose that $k>2$. Recall from \ref{325} 
that $\sigma_i$ is a flip
for every $i\in\{0,\dotsc,k-1\}\smallsetminus\{i_1,i_2\}$, so 
equivalently we are assuming that the Mori program \eqref{veronica}
 contains some flip.

We define an integer $m\in\{k-3,k-2,k-1\}$ and a morphism
 $\eta\colon X_{m+1}\to X_k$ as follows:
\begin{enumerate}[$\bullet$]
\item if $\sigma_{k-1}$ is a flip, set 
$m:=k-1$ and $\eta:=\text{Id}_{X_k}$;
\item if $\sigma_{k-1}$ is a contraction and $\sigma_{k-2}$ is a flip, set
$m:=k-2$ and $\eta:=\sigma_{k-1}\colon X_{k-1}\to X_k$;
\item if both $\sigma_{k-1}$ and $\sigma_{k-2}$ 
are contractions, set $m:=k-3$
and
$\eta:=\sigma_{k-2}\circ\sigma_{k-1}\colon X_{k-2}\to X_k$.
\end{enumerate}
It follows from these definitions that $Q_m$ is a small extremal ray of $X_m$ 
and $\sigma_m\colon
X_m\dasharrow X_{m+1}$ is a flip;
let $Q'_{m+1}$ 
be the corresponding small extremal ray of $X_{m+1}$. Set moreover
$\widetilde{\ph}:=\ph\circ\eta\colon
 X_{m+1}\to Y$.
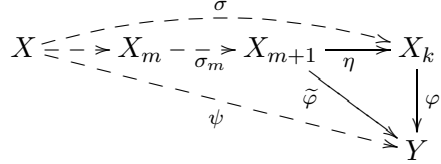
$$\xymatrix{
X\ar@/^1pc/@{-->}[rrr]^{\sigma}\ar@{-->}[r]\ar@{-->}[drrr]_{\psi}&
{X_m}\ar@{-->}[r]_{\sigma_m} &
{X_{m+1}}\ar[r]_{\eta}\ar[dr]_(.3){\widetilde{\ph}} &
{X_k}\ar[d]^{\ph}\\
&&&Y
}$$

We keep the same notations as in  the proof of
 Lemma~\ref{conicbdl}; in particular we set $T_i:=\sigma(E_i)\subset
 X_k$ for $i=1,2$. Notice that when $m=k-2$ (respectively, $m=k-3$), 
$\eta$ is just the smooth blow-up of $T_2\subset X_k$ (respectively, 
of $T_1\cup T_2\subset X_k$).

Recall that for $j=1,2$ we have $Q_{i_j}\not\subset\N(D_{i_j},X_{i_j})$, 
in particular $\sigma_{i_j}$ is finite on $D_{i_j}$; this implies that $\eta$ 
is finite on $D_{m+1}$.

We show that every fiber of $\widetilde{\ph}$ has dimension
$1$. Indeed this is true for $\ph$  
by \ref{conicbdl}$(1)$. Moreover $\eta$ is an isomorphism 
over  $X_k\smallsetminus(T_1\cup T_2)$, therefore  $\widetilde{\ph}$
has one-dimensional fibers over $Y\smallsetminus \ph(T_1\cup T_2)$. 
On the other hand, we know by  \ref{conicbdl}(3) that there exist
open subsets $U\subseteq X$ and $V\subseteq Y$ such that $\ph(T_1\cup
T_2)\subset  
V$,
 both $\psi\colon U\to V$ and $\ph_{|\ph^{-1}(V)}\colon \ph^{-1}(V)\to
 V$ are conic bundles, and $\sigma_{|U}\colon U\to \ph^{-1}(V)$ is
 just the blow-up of $T_1$ and $T_2$. This
 implies that $\widetilde{\ph}_{|\widetilde{\ph}^{-1}(V)}\colon  
\widetilde{\ph}^{-1}(V)\to V$ is a conic bundle, in particular it has
one-dimensional fibers over $\ph(T_1\cup T_2)\subset V$. 

Recall from \ref{conicbdl}$(1)$ that $\ph$ is finite on $D_k=\eta(D_{m+1})$,
and since $\eta$ is finite on $D_{m+1}$, we deduce that
 $\widetilde{\ph}$ must be finite on $D_{m+1}$. Notice also that
$D_{m+1}\supset A_{m+1}\supseteq \Lo(Q_{m+1}')$ (see
\ref{general}$(4)$). 
As in the proof of Lemma~\ref{conicbdl},
using \cite[Lemma~3.8]{31} 
we see that every
fiber of $\widetilde{\ph}$ which intersects $\Lo(Q_{m+1}')$ is an
integral rational curve. 

Let $C\subset X_{m+1}$ be an irreducible curve with $[C]\in Q_{m+1}'$,
and set $S:=\widetilde{\ph}^{-1}(\widetilde{\ph}(C))$, so that $S$ is
an irreducible surface.

Since, by \ref{327}, $F$ and $\widehat{F}$ have positive
intersection with a general fiber of $\psi$ in $X$, 
 $F_{m+1}$ and $\widehat{F}_{m+1}$ have positive intersection with every
fiber of $\widetilde{\ph}$ in $X_{m+1}$. 
In particular,  $F_{m+1}$ and $\widehat{F}_{m+1}$ intersect $S$.

On the other hand by \ref{328} the divisors $F_m$ and $\widehat{F}_{m}$ in $X_m$
are disjoint from  $\Lo(Q_{m})$, therefore $F_{m+1}$ and $\widehat{F}_{m+1}$
are disjoint from  $\Lo(Q_{m+1}')$. We deduce
 that:  
\renewcommand{\theequation}{\thethmdue}\stepcounter{thmdue}
\begin{equation}\label{isterico}
F_{m+1}\cap C=\widehat{F}_{m+1}\cap C=\emptyset\ \text{ and
}\ \dim(F_{m+1}\cap S)=\dim(\widehat{F}_{m+1}\cap S)=1. 
\end{equation}

For $i=1,2$ call $G_i$ the image of $E_i$ in $X_{m+1}$, so that
$T_i=\eta(G_i)$ and $\ph(T_i)=\widetilde{\ph}(G_i)$. Notice that 
$A_k\smallsetminus(T_1\cup T_2)=\eta(A_{m+1}\smallsetminus(G_1\cup G_2))$.

Recall that the open subset $V\subseteq Y$ was defined in \eqref{V} as
$$V:=Y\smallsetminus \ph\left(A_k\smallsetminus(T_1\cup T_2)\right)=
Y\smallsetminus \widetilde{\ph}\left(A_{m+1}\smallsetminus(G_1\cup G_2)
\right).$$
By \ref{pluto}$(1)$ and \ref{conicbdl}$(2)$ we have $\Lo(Q_{m+1}')\cap
(G_1\cup G_2)=\emptyset$. In particular $C\subseteq
\Lo(Q_{m+1}')\subseteq  
A_{m+1}\smallsetminus(G_1\cup G_2)$, thus $$\widetilde{\ph}(C)\subseteq 
Y\smallsetminus V.$$

On the other hand we also have
$\widetilde{\ph}(G_1\cup G_2)=\ph(T_1\cup T_2)\subset V$,
therefore we deduce that
$\widetilde{\ph}(G_1\cup G_2)\cap\widetilde{\ph}(C)=\emptyset$ and hence
$$
(G_1\cup G_2)\cap S=\emptyset.
$$

Finally  by \ref{326b} we have $F\cap\widehat{F}=\emptyset$ in $X$,
and by \ref{328} the divisors $F$ and $\widehat{F}$ are disjoint from
the locus of every flip in the Mori program \eqref{veronica}. This
implies that 
 $F_{m+1}\cap\widehat{F}_{m+1}\subseteq G_1\cup G_2$, therefore:
$$F_{m+1}\cap\widehat{F}_{m+1}\cap 
S=\emptyset.$$ 
Together with \eqref{isterico}, this yields that $C$, $F_{m+1}\cap S$,
and $\widehat{F}_{m+1}\cap S$ are pairwise 
disjoint
curves in $S$.

Let $C'$ be an irreducible component of $\widehat{F}_{m+1}\cap
S$. Since $\widetilde{\ph}_{|S}\colon S\to\widetilde{\ph}(C)$ 
is a fibration in integral rational curves, we have
$C'\equiv \lambda C+\mu f$ where $\lambda,\mu\in\R$ and $f\subset S$
is a fiber. Then
$0=F_{m+1}\cdot C'= \mu F_{m+1}\cdot f$ while $F_{m+1}\cdot f>0$, hence $\mu=0$ and
$[C']\in Q'_{m+1}$. Therefore $C'\subseteq\Lo(Q'_{m+1})\cap F_{m+1}$,
a contradiction because $\Lo(Q'_{m+1})\cap F_{m+1}=\emptyset$. 
\end{parg2}
\begin{parg2}\label{330}
Since $k=2$, $X_2$ is smooth and $\sigma\colon X\to X_2$ is just the
blow-up of two disjoint smooth subvarieties $T_1,T_2\subset X_2$, of
codimension $2$. In fact, we have $A_2=T_1\cup T_2$ (see
\ref{general}$(4)$), and by \eqref{V}
the description in \ref{conicbdl}$(3)$ and
\ref{conicbdl}$(4)$ 
holds 
  with $V=Y$ and $U=X$. In particular, $Y$ is smooth, 
$\ph\colon X_2\to Y$ and $\psi\colon X\to Y$ are conic bundles, 
  $\rho_X-\rho_Y=3$, and  the divisors
$Z_1=\psi(E_1)$ and $Z_2=\psi(E_2)$ are disjoint in $Y$. 
Moreover we have
$\psi(F)=Y$  by \ref{327}.

The situation is very similar to the case where $\ph$ is finite on
$(E_0)_s$ in \ref{second}, 
with the difference that the
$E_i$'s do not need to be products. In the same way we use
Lemma~\ref{Vequivalence} to show that $[\psi(l)]\in\NE(Y)$ belongs to an
extremal ray of $Y$, whose contraction is a smooth conic bundle
$\zeta\colon Y\to Y'$, finite on $Z_1$ and $Z_2$; in particular 
$Y'$ is smooth of dimension $n-2$.
The contraction $\psi':=\zeta\circ\psi\colon X\to Y'$ is
equidimensional and hence flat, and $\rho_X-\rho_{Y'}=4$.
Moreover the general fiber of $\psi'$ is a Del Pezzo surface $S$
containing curves $f_1,\widehat{f}_1,f_2,\widehat{f}_2,l$, hence
$\N(S,X)=\ker(\psi')_*$ and $\psi'$ is quasi-elementary.
$$\xymatrix{
 X\ar[d]_{\psi'}\ar[dr]^{\psi}\ar[r]^{\sigma} & {X_{2}}\ar[d]^{\ph} \\
{Y'} & Y\ar[l]^{\zeta}
}$$
\end{parg2}
\begin{parg2}\label{331}
We show that the conic bundle $\ph\colon X_2\to Y$ is smooth.

\medskip

 By contradiction, suppose that
this is not the case, and let $\Delta_{\ph}\subset Y$ be 
 the
discriminant divisor of $\ph$. Recall that this is an effective,
reduced divisor in $Y$ such that $\ph^{-1}(y)$ is singular if and only
if $y\in\Delta_{\ph}$. 

Consider also the discriminant
divisor $\Delta_{\psi}\subset Y$ of the 
conic bundle $\psi\colon X\to Y$. Since $\ph$ is smooth over $Z_1$ and $Z_2$,
the divisors $\Delta_{\ph},Z_1,Z_2$ are pairwise disjoint, and
$\Delta_{\psi}=\Delta_{\ph}\cup Z_1\cup Z_2$. 

The fibers of $\psi$ over $Z_1\cup Z_2$ are singular but reduced,
hence $\psi^{-1}(y)$ is non-reduced if and only if $\ph^{-1}(y)$ is.
Let $W\subset\Delta_{\ph}$
be the set of points $y$ such that $\psi^{-1}(y)$ (equivalently,
$\ph^{-1}(y)$)
is non-reduced. Then
$W$ is a closed subset of $Y$, and
$W\subseteq\Sing(\Delta_{\ph})$ (see for instance
\cite[Proposition 1.8(5.c)]{sarkisov}). Moreover by
\cite[Proposition 4.3]{wisn} we know that $-K_Y\cdot C>0$ for every 
 irreducible curve
 $C\subset Y$ not contained in $W$.

 For $i=1,2$ we have $\codim\N(Z_i,Y)\leq 1$, because $\zeta(Z_i)=Y'$
 and hence $\zeta_*(\N(Z_i,Y))=\N(Y')$.
This
 yields
$Z_1^{\perp}=Z_2^{\perp}=\Delta_{\ph}^{\perp}=
\N(Z_1,Y)=\N(Z_2,Y)$ (see Remark \ref{elem}).
The three divisors $\Delta_{\ph},Z_1,Z_2$
are numerically proportional, nef, and cut a facet
of $\NE(Y)$, whose contraction
$\beta\colon
Y\to\pr^1$
sends $\Delta_{\ph},Z_1,Z_2$ to points (see \cite[Lemma~2.6]{fanos}). 
Even if a priori we do not know whether
every curve contracted by $\beta$ has positive anticanonical degree,
the general fiber of $\beta$ does not meet $W$, therefore it
is a Fano manifold. Moreover
$\NE(\beta)$ is generated by finitely many classes of rational
curves (see \cite[Lemma~2.6]{fanos}). 
Thus the same proof as \cite[Lemma~4.9]{31} yields
 that $Y\cong\pr^1\times Y'$, and  $\Delta_{\ph}=\{pts\}\times Y'$. 

In particular $\Delta_{\ph}$ is smooth, hence $W=\emptyset$ and $Y$
is Fano. Because $Y\cong\pr^1\times Y'$,
$Y'$ is Fano too, so that 
each connected component of $\Delta_{\ph}$ is
simply connected. However this is impossible, because by a standard
construction the conic bundle $\ph$ defines a double cover of every
irreducible component of  $\Delta_{\ph}$, obtained by considering the
components of the fibers in the appropriate Hilbert scheme of lines,
see \cite[\S1.5]{beauville}
and \cite[\S1.17]{sarkisov}. Since $\ph$ is an elementary contraction,
this double cover is non-trivial; on the other hand
it is also \'etale, because every fiber of $\ph$ is reduced, and we
have a contradiction.
\end{parg2}
\begin{parg2}
Since $\ph\colon X_2\to Y$ is smooth,
 every fiber of the conic bundle $\psi\colon X\to Y$ is
reduced. Then
 \cite[Proposition 4.3]{wisn} shows
that $Y$ and $Y'$ are Fano.
Finally $c_{Y'}\leq 3$ by the following Remark, which concludes the
proof of Proposition \ref{secondameta}.
\end{parg2}
\vspace{-17pt}
\end{proof}
\begin{remark2}\label{codimY}
Let $X$ be a Fano manifold, $\ph\colon X\to Y$ a
surjective morphism, and $D\subset X$ a prime divisor. We have
$\N(\ph(D),Y)=\ph_*(\N(D,X))$, hence 
 $\codim\N(D,X)\geq\codim\N(\ph(D),Y)$.
In particular, if $Y$ is a Fano manifold, then $c_Y\leq c_X$.
\end{remark2}
\section{Applications}\label{quarta}
In this final section we prove the results
stated in the introduction.
\begin{proof}[Proof of Theorem \ref{product}]
We have $c_X\geq\codim\N(D,X)\geq 3$. If $c_X=3$, Theorem \ref{gen} 
yields $(ii)$.
If instead $c_X\geq 4$,
applying
iteratively 
Theorem \ref{gen}, we can write $X=S_1\times\cdots\times S_r\times Z$,
where $S_i$ are Del Pezzo surfaces, $r\geq 1$, and $Z$ is a Fano
manifold with $c_Z\leq 3$. 

If $D$ dominates $Z$ under the projection, up to reordering
$S_1,\dotsc,S_r$ we can assume that $D$ dominates
$S_2\times\cdots\times S_r\times Z$. Then
$\codim\N(D,X)\leq\rho_{S_1}-1$ (see Example \ref{max}), and
we get $(i)$.

Suppose instead that $D=S_1\times\cdots\times S_r\times D_Z$, where
$D_Z\subset Z$ is a prime divisor. Then
$$3\geq c_Z\geq\codim\N(D_Z,Z)=\codim\N(D,X)\geq 3,$$
and the inequalities above are equalities. Therefore again by
Theorem \ref{gen} we have a flat, 
quasi-elementary contraction $Z\to W$, where $W$ is a
Fano manifold with $\dim W=\dim Z-2$, and $\rho_Z-\rho_W=4$. 
Then the induced contraction
$X\to S_1\times\cdots\times S_r\times W$ satisfies $(ii)$.
\end{proof}
\begin{proof}[Proof of Corollary \ref{lorenzo}]
We have $c_X\geq\codim\N(D,X)\geq 3$.
Suppose that $X$ is not a product of a Del Pezzo surface with another
variety. Then 
Theorem \ref{gen} shows that
 $c_X= 3$ and   there is a quasi-elementary
contraction $X\to T$ where $T$ is a Fano manifold, $\dim
T=n-2$, and $\rho_X-\rho_T= 4$.
If $n=4$, \cite[Theorem 1.1]{fanos} implies that $\rho_T\leq 2$, hence
$\rho_X\leq 6$.
The case $n=5$ follows similarly.
\end{proof}
\begin{lemma}\label{corollZ}
Let $X$ be a Fano manifold, $D\subset X$ a
prime divisor, and $\ph\colon X\to Y$ a
contraction.
Then 
$\codim\N(\ph(D),Y)\leq 8$. 

Suppose moreover that
 $\codim\N(\ph(D),Y)\geq 4$. Then
  $X\cong
S\times T$ and $Y\cong W\times Z$, where $S$ is a Del Pezzo
surface, $W$ is a blow-down of $S$, and $\codim\ph(D)\leq 2$. 
More precisely, one of the following holds:
\begin{enumerate}[$(i)$]
\item $\ph(D)$ is a divisor in $Y$, and dominates $Z$ under the
 projection;
\item $\ph(D)=\{p\}\times Z$ and $D=C\times T$, where
 $C\subset S$ is a curve contracted to $p\in W$.
\end{enumerate}
\end{lemma}
\begin{proof}
We have $\codim\N(\ph(D),Y)\leq \codim\N(D,X)\leq 8$
by Remark \ref{codimY}
and Theorem \ref{product}. 
Suppose that
 $\codim\N(\ph(D),Y)\geq 4$. Then, again by  Theorem \ref{product},
 $X\cong S\times T$ where $S$ is a Del
Pezzo surface,  and $D$ dominates $T$ under the projection. Therefore $Y\cong
W\times Z$, $\ph$ is induced by two contractions $S\to W$ and $f\colon
T\to
Z$, and $\ph(D)$ dominates $Z$ under the projection. 

In particular
$\dim W\leq 2$
and $\dim\N(\ph(D),Y)\geq \rho_Z$, hence $\rho_W\geq
\codim\N(\ph(D),Y)\geq 4$. 
 This implies that $\dim W=2$, thus
$W$ is a blow-down
of $S$, and $\ph(D)$ has codimension $1$ or $2$ in $Y$.

If $\ph(D)$ is a divisor, we have $(i)$.
Suppose that $\codim\ph(D)=2$, and consider the factorization of $\ph$
as $S\times T\stackrel{\psi}{\to}W\times T\stackrel{\xi}{\to}
W\times Z$. Then $\xi=(\text{Id}_W,f)$
induces an isomorphism $W\times\{t\}\to W\times\{f(t)\}$ for every
$t\in T$. 
If $t$
is general, we have $\dim \ph(D)\cap
(W\times\{f(t)\})=0$ and
$\psi(D)\cap (W\times\{t\})\cong \ph(D)\cap
(W\times\{f(t)\})$. This implies that $\psi(D)$ has codimension $2$ in
$W\times T$, hence $D$ is an exceptional divisor of $\psi$, which
gives the statement.
\end{proof}
Corollary \ref{p1} follows from
Theorem \ref{product}, while
Corollary \ref{divtocurve} is a straightforward application of Lemma
\ref{corollZ} (just take the Stein factorization of $\ph$).
\begin{proof}[Proof of Corollary \ref{3folds}]
By taking the Stein factorization, we can assume that $\ph$ is a contraction; we also assume that $\rho_Y\geq 6$.
By \cite[Lemma~2.6]{fanos} we know that $Y$ has some elementary contraction $\psi\colon Y\to Z$, and $\dim Z\geq 2$ because $\rho_Z\geq 5$.

We define a prime divisor $D\subset X$ depending on $\psi$, as follows.
If $\dim Z=2$, let $D\subset X$ be any prime divisor such that $\dim\psi(\ph(D))=1$.
If $\psi$ is birational and divisorial, let $D\subset X$ be a prime divisor such that $\ph(D)\subseteq\Exc(\psi)$.
If $\psi$ is birational and small, its lifting in
$X$ (see \cite[\S~2.5]{fanos}) must
be an elementary contraction of type $(n-1,n-2)^{sm}$; let $D$ be its exceptional divisor: then again we have  $\ph(D)\subseteq\Exc(\psi)$.

In any case $\dim\N(\ph(D),Y)\leq 2$, so that Lemma \ref{corollZ} 
implies the statement.
\end{proof}
\footnotesize
\providecommand{\bysame}{\leavevmode\hbox to3em{\hrulefill}\thinspace}
\providecommand{\MR}{\relax\ifhmode\unskip\space\fi MR }
\providecommand{\MRhref}[2]{%
  \href{http://www.ams.org/mathscinet-getitem?mr=#1}{#2}
}
\providecommand{\href}[2]{#2}

\bigskip

\noindent C.\ Casagrande\\
Dipartimento di Matematica, Universit\`a di Pavia \\
via Ferrata 1 \\
 27100 Pavia - Italy 

\smallskip

\noindent \emph{Current address:}\\
Dipartimento di Matematica, Universit\`a di Torino \\
via Carlo Alberto 10 \\
10123 Torino - Italy\\
cinzia.casagrande@unito.it
\end{document}